\documentclass[11pt, a4paper]{article}

\usepackage[T1]{fontenc}
\usepackage[utf8]{inputenc}
\usepackage{lmodern}
\usepackage{microtype}

\usepackage[left=2.56cm, right=2.56cm, top=2.25cm, bottom=3.25cm]{geometry}
\usepackage{amsmath}
\usepackage{amssymb}
\usepackage{amsthm}
\usepackage{mathtools}
\usepackage{bbm}
\usepackage{enumerate}
\usepackage{enumitem}
\usepackage[dvipsnames]{xcolor}

\usepackage{pdfsync}

\usepackage{multirow}
\usepackage{empheq}
\usepackage{color}
\usepackage[%
    font={small,sf},
    labelfont=bf,
    format=hang,    
    format=plain,
    margin=0pt,
    width=0.8\textwidth,
]{caption}
\usepackage{subcaption}
\usepackage{placeins}
\usepackage[square,comma,numbers]{natbib}
\usepackage{hyperref}
\usepackage{graphicx} %

\usepackage{tikz}
\usetikzlibrary{babel}

\numberwithin{equation}{section}
\usepackage{environ}
\makeatletter
\newsavebox{\measure@tikzpicture}
\NewEnviron{scaletikzpicturetowidth}[1]{%
  \def\tikz@width{#1}%
  \begin{lrbox}{\measure@tikzpicture}%
  \BODY
  \end{lrbox}%
  \pgfmathparse{#1/\wd\measure@tikzpicture}%
  \BODY
}
\makeatother

\usepackage[dvipsnames]{xcolor}
\usepackage{soul}

\newcommand*\indic[1]{\mathbbm{1}_{\{ #1 \}}}
\newcommand*\indica[1]{\mathbbm{1}_{ #1 }}
\def\cA{{\mathcal A}}
\def\cB{{\mathcal B}}

\def\cF{{\mathcal F}}

\def\cL{{\mathcal L}}

\def\cP{{\mathcal P}}

\def\E{\mathbb{E}}

\def\N{\mathbb{N}}

\def\Q{\mathbb{Q}}
\def\P{\mathbb{P}}
\def\R{\mathbb{R}}

\def\d{\mathrm{d}}

\providecommand{\abs}[1]{\lvert#1\rvert}
\providecommand{\norm}[1]{\lVert#1\rVert}

\usepackage{MnSymbol}
\theoremstyle{plain}
\newtheorem{theorem}{Theorem}[section]
\newtheorem{lemma}[theorem]{Lemma}
\newtheorem{corollary}[theorem]{Corollary}
\newtheorem{proposition}[theorem]{Proposition}
\newtheorem{assumption}[theorem]{Assumption}
\newtheorem{definition}[theorem]{Definition}

\newtheorem{remark}[theorem]{Remark}

\usepackage[symbol]{footmisc}

\date{\vspace{-1em}\normalsize{\today}}
\title{Optimal control under unknown intensity with Bayesian learning}

\author{Nicolas Baradel\thanks{Inria, CMAP, CNRS, \'Ecole polytechnique, Institut Polytechnique de Paris, 91120 Palaiseau, France.} 
\and Quentin Cormier\footnotemark[1]}

\date{\today}
\begin{document}
\maketitle
\vspace{5pt}

\abstract{
We investigate an optimal control problem motivated by neuroscience, where the dynamics is driven by a Poisson process with a controlled stochastic intensity and an unknown parameter. 
 Given a prior distribution for the unknown parameter, we describe its evolution using Bayes' rule. We reformulate the optimization problem by applying Girsanov's theorem and establish a dynamic programming principle. Finally, we characterize the value function as the unique viscosity solution to a finite-dimensional Hamilton-Jacobi-Bellman equation, which can be solved numerically.
}
\vspace{10pt}
\vspace{10pt}
\vspace{10pt}

\section{Introduction}

Consider the following stochastic integrate-and-fire neuron model \cite{eva1, eva2, CTV}, where $(Y_t)$ represents the membrane potential of a neuron.
The neuron emits ``spikes'' randomly at rate $f_\Lambda(Y_t)$. When a spike occurs, the potential resets to zero. Between the spikes, the dynamics is given by the ordinary differential equation
\[ \dot{Y_t} = b(Y_t) + \gamma_t. \]
In this model, the functions $(\lambda, y) \mapsto f_\lambda(y) \in \R_+$ and $y \mapsto b(y) \in \R$ are deterministic and known. 
The parameter $\Lambda$ is a random variable with a known distribution $\mu$, supported on $\R_+$, but its realization is unknown.
Our objective is to estimate the value of $\Lambda$ as accurately as possible. The input current $(\gamma_t)$ serves as a control and can be chosen to improve the estimation.
Crucially, we only observe the spike times of $(Y_t)$, and thus $(\gamma_t)$ must depend causally on the spike times up to time $t$. The central question addressed in this article is:
\begin{center}
	\emph{How should the stochastic control $(\gamma_t)$ be chosen to optimally estimate the unknown parameter?}
\end{center}
This leads us to formulate an optimal control problem driven by a counting process with unknown intensity. Our approach extends beyond neuroscience to applications requiring optimal control of processes with unknown parameters, such as in online learning scenarios. We present a mathematical framework to address this class of problems.

\medbreak

Such problems have been extensively studied in the discrete-time stochastic optimal control literature (see, e.g., \cite{easley1988controlling, hernandez2012adaptive}). A key approach involves assigning a prior distribution to the unknown parameter, which is updated at each time step using Bayes' rule. This Bayesian framework is central to filtering theory; see \cite{crisan} for a historical account of filtering theory and \cite{MR3075395} for an overview of classical filtering techniques. In this work, we combine insights from filtering theory and stochastic control.

In the continuous-time setting, particularly within the impulse control framework, these issues have been addressed by \cite{baradel2018optimal} in a Brownian framework and extended by \cite{baradel2024optimal} to a Poisson framework. As in the discrete-time setting, a Bayesian approach is employed: the unknown parameter is sampled from a \emph{prior distribution}, and a stochastic process in the space of probability measures, known as the \emph{posterior distribution}, is introduced. This approach leads to a dynamic programming principle and a characterization of the value function via a quasi-variational parabolic equation, interpreted in the sense of viscosity solutions.

\medbreak
In \cite{lions2016} (lectures given at \emph{Coll\`{e}ge de France}) and \cite{bismuth2019portfolio} (with applications to asset management), the authors consider instances of the following diffusive model:
\[ d X_t = \Lambda b(X_t, \gamma_t) \d t + \sigma(X_t) \d W_t. \]
In this model, the process $(X_t)$ is observed, but the parameter $\Lambda$ is unknown. A key idea introduced in these works is to reformulate the problem using Girsanov's theorem. This approach eliminates the drift and explicitly accounts for the dependence on the unknown parameter through the exponential martingale in Girsanov's transformation. Furthermore, in cases where $\Lambda$ appears linearly or for specific classes of prior distributions known as conjugate families, the authors derive a finite-dimensional Hamilton-Jacobi-Bellman (HJB) equation satisfied by the value function.

A specificity of our setting is the following: the conditional distribution of the unknown parameter evolves continuously and exhibits jumps; these jumps in the conditional distribution occur precisely at the spike times of the underlying point process, with an intensity that depends on the unknown parameter. 
There is an extensive literature on filtering: in our case, the filtering problem without control can be solved by standard methods. 

The main difficulty lies in the coupling between the stochastic control problem and the filtering problem. Indeed, modifying the control $(\gamma_t)$ alters the spike times of the process. 
In addition, since the controller observes only the spike times, a lot of care is required to define the admissible controls. In particular, it is generally not possible to use the same control $(\gamma_t)$ for two different initial conditions $y$ and $y'$ of $(Y_t)$. Indeed, changing the initial condition may alter the spike times, rendering a control valid for $y$ invalid for $y'$. This makes particularly challenging to obtain a priori regularity estimates for the value function. As in \cite{lions2016, bismuth2019portfolio}, Girsanov's theorem plays a crucial role in addressing these difficulties. 

Another challenge stems from the infinite-dimensional nature of the measure space, where the posterior distribution evolves continuously over time. In cases where the intensity of the underlying point process depends linearly on the unknown parameter, we demonstrate that the posterior distribution is confined to a two-dimensional subspace of $\mathbb{R}^{2}$. Furthermore, we establish that the value function associated with the optimal control problem is the unique viscosity solution to a specific finite-dimensional Hamilton-Jacobi-Bellman (HJB) equation.

\medbreak

We now briefly outline our main contributions and the organization of this paper.

The control problem is introduced in Section~\ref{sec:def-problem}, where we characterize the admissible controls and demonstrate that the optimization problem is independent of the initial probability space. In Section~\ref{sec:girsanov-reformulation}, we reformulate the problem using Girsanov's theorem for point processes with stochastic intensity and establish the equivalence between the original and reformulated problems. In Section~\ref{sec:the-value-function}, we introduce the value function, and in Section~\ref{sec:reduction-finite-dim}, we demonstrate how to obtain a finite-dimensional representation of the value function when the function $f_\lambda(y)$ is linear in $\lambda$, i.e. $f_\lambda(y) = \lambda g(y)$. This result is analogous to the finite-dimensional reduction obtained in \cite{lions2016, bismuth2019portfolio} for the diffusive setting with linear drift. Thereafter, we focus our analysis on this finite-dimensional problem.

In Section~\ref{sec:space-reg}, we prove that the value function is Lipschitz continuous with respect to the spatial variables. To this end, we conduct a detailed analysis of the exponential martingale involved in Girsanov's transformation. The key difficulty is that the admissible controls are not assume to be bounded. Using these regularity estimates, we establish a dynamic programming principle with deterministic stopping times in Section~\ref{sec:PPD-deterministic-times}. Leveraging this result, we demonstrate in Section~\ref{sec:time-regularity} that the value function is continuous with respect to time. In Section~\ref{sec:PPD}, we provide a general dynamic programming principle that accommodates arbitrary stopping times.

In Section~\ref{sec:visc}, we show that the value function is the unique viscosity solution to a specific Hamilton-Jacobi-Bellman (HJB) equation, given by equation~\eqref{eq:HJB-reformulated}. This first-order HJB equation includes zero-order terms arising from the jumps in the underlying dynamics. We adopt the methodology of \cite{MR861089} to address these terms. A particular challenge in proving the comparison result arises due to the unbounded nature of both the controls and the spatial variables.

Finally, in Section~\ref{sec:examples}, we present several numerical illustrations and discuss the behavior of the optimal control.

\section{The control problem}
\label{sec:control-problem}
\subsection{The optimization problem}
\label{sec:def-problem}
Let $\mu \in \cP_2(\R_+)$ be a probability measure supported on $\R_+$ with a finite second moment.
Consider functions $b$ and $f$ satisfying the following assumption:
\begin{assumption}
	\label{ass:initial-ass}
	The function $b: \R \rightarrow \R$ is Lipschitz, and $f: \R_+ \times \R \rightarrow \R_+$ is continuous and bounded for each $\lambda \geq 0$, with quadratic growth in $\lambda$:
	\[ \exists C > 0, \quad \sup_{y \in \R} f_\lambda(y) < C(1 + \lambda^2). \]
\end{assumption}

Let $(\Omega, \cF, (\cF_t), \P)$ be a filtered probability space satisfying the usual conditions.
Let $\Pi(\d u, \d z)$ be an $(\cF_t)$-Poisson measure with intensity given by the Lebesgue measure $\d u \d z$ on $\R_+ \times \R_+$. Let $\Lambda$ be an $\cF_0$-measurable random variable with probability distribution $\mu$.
Fix a time horizon $T > 0$. Let $(\gamma_{t})_{t \in [0, T]}$ be an $(\cF_t)$-progressively measurable process such that
\[ \P(\d \omega) a.s, \quad \gamma \in L^2([0, T]). \]
We consider the following stochastic differential equation for all $t \in [0, T]$:
\begin{align*}
	Y_t &= y + \int_0^t (b(Y_u) + \gamma_u)\d u - \int_0^t \int_{\R_+} Y_{u-} \indic{z \leq f_\Lambda(Y_{u-})}  \Pi(\d u, \d z). 
\end{align*}
This stochastic differential equation, driven by the Poisson measure $\Pi$, admits a unique pathwise solution.  Indeed, it suffices to solve the equation between consecutive jumps; see \cite[Ch. IV, 9]{MR637061}. We define:
\begin{equation}
	\label{eq:def-of-Nt}
N_t = \int_0^t \int_{\R_+} \indic{z \leq f_\Lambda(Y_{u-})} \Pi(\d u, \d z), \end{equation}
as the number of events between $0$ and $t$, and denote its canonical filtration by:
\[ \cF^{N}_t = \sigma \{ N_s, ~s\leq t \}. \]
\begin{definition}
	A control $\gamma$ is admissible if it is predictable with respect to the filtration $(\cF^N_t)_{t \geq 0}$.
\end{definition}
We denote by $\cA$ the set of admissible controls. 
The posterior distribution of $\Lambda$ given the observed spike times is defined as:
\[ M^{\gamma}_t := \cL(\Lambda ~|~ \cF^{N}_t). \]
Let $\kappa > 0$. The optimization problem is formulated as follows:
\begin{equation}  
	\label{eq:original-value}
	\inf_{\gamma \in \cA} \E \left[  \frac{1}{2}\int_0^T { \gamma^2_s \d s} + \kappa \text{Var}(M^{\gamma}_T) \right].  
\end{equation}
In other words, our goal is to minimize the variance of the posterior distribution, with the quadratic cost on the control serving as a regularization to prevent excessively large controls. A key difficulty in this formulation is that the set of admissible controls is not known a priori. Indeed, we require the control $\gamma$ to be predictable with respect to $(\cF^N_t)$, where $N_t$ is itself determined by solving a stochastic differential equation driven by a Poisson measure and $\gamma$.
In the next section, we provide a more explicit characterization of the admissible controls.
\subsection{Admissible controls and structure of the optimization problem}
We now reformulate the optimization problem to demonstrate that the value function defined by~\eqref{eq:original-value} is independent of the probability space $(\Omega, \cF, (\cF_t), \P)$.
Let $(\tau_k)_{k \geq 1}$ denote the successive jump times of $(N_t)_{t \geq 0}$, with the convention that $\tau_0 = 0$.
The following proposition characterizes the admissible controls. 
\begin{proposition}
	\label{prop:caracterisation-progressible-processes}
Let $\gamma$ be a process such that  $\P$ a.s., $\gamma \in L^2([0, T])$.
In order for $\gamma$ to be $(\cF^N_t)$-predictable, it is necessary and sufficient that it admits the representation
\[  \gamma_t = \sum_{k \geq 0} \Gamma_{k}(t) \indica{(\tau_{k}, \tau_{k+1}]}(t), \]
where for all $k \in \N$, the mapping $(\omega, t) \mapsto \Gamma_k(\omega)(t)$ is $\cF^N_{\tau_k} \otimes \cB([0, T]) $-measurable, where $\cB([0, T])$ denotes the Borelian $\sigma$-field of $[0, T]$.
\end{proposition}
\begin{proof} See \cite[A2, Th. 34]{BP}.
 \end{proof}
 As $L^2([0, T])$ is a separable Hilbert space, it holds that the application $\omega \mapsto \Gamma_k(\omega)$ which maps $\Omega$ to $L^2([0, T])$ is
 measurable from $\cF^N_{\tau_k}$ to $\cB(L^2([0, T]))$.
 Consider $D([0, T])$, the Skorokhod space of c\`{a}dl\`{a}g functions on $[0, T]$. By Doob's lemma~\cite[Lemma~1.13]{Kallenberg}, there exists a deterministic and measurable function $\Psi_k: D([0, T]) \rightarrow L^2([0, T])$ such that:
\[ \Gamma_k = \Psi_k(N_{\cdot \wedge \tau_k}).  \]
Here, we use the notation $a \wedge b := \min(a, b)$.
This establishes a one-to-one correspondence between an admissible control $\gamma$
and a collection of measurable functions $(\Psi_k)_{k \geq 0}$. 
We denote by $\tilde{A}$ the set of functions $(\Psi_k)_{k \geq 0}$, such that each $\Psi_k$ is a measurable function from $D([0, T])$ to $L^2([0,T])$.
\begin{lemma}
	\label{lem:triple-L-N-Y}
	Let $\Psi \in \tilde{A}$. Let $\Lambda$ be a $\cF_0$-measurable random variable with probability distribution $\mu$. 
	Then the following equation has a unique pathwise solution:
	\begin{align*}	
	Y_t &= y + \int_0^t (b(Y_u) + \gamma_u)\d u - \int_0^t \int_{\R_+} Y_{u-} \indic{z \leq f_\Lambda(Y_u-)}  \Pi(\d u, \d z), \\
		N_t &= \int_0^t \int_{\R_+} \indic{z \leq f_\Lambda(Y_{u-})} \Pi(\d u, \d z), \\
		\gamma_t &= \sum_{k \geq 0} \Psi_k(N_{\cdot \wedge \tau_k})(t) \indica{(\tau_k, \tau_{k+1}]}(t). 
	\end{align*}
	These equations characterize the law of $(N_t)_{t \in [0, T]}$. We write:
	\[ Q_{\Psi, \mu} := \cL((N_t)_{t \in [0, T]}). \]
\end{lemma}
\begin{proof}
	The main idea of the proof is that the equation can be solved recursively from each jump instant $\tau_k$ to the next jump instant $\tau_{k+1}$ and the equation reduces to the ODE
	\[ Y_{t \wedge \tau_{k+1}} = Y_{\tau_k} + \int_{\tau_k}^{t \wedge \tau_{k+1}} [b(Y_u) + \Gamma_k(u)] \d u, \]
	where $\Gamma_k(u) = \Psi_k(N_{\cdot \wedge \tau_k})(u)$.
	In addition, 
	\[ \E \left[N_t\right] = \int_0^t \E \left[f_\Lambda(Y_s)\right] \d s  \leq C T\left( 1 + \E\left[ \Lambda^2\right]\right) < \infty, \]
	therefore the jump instants do not accumulate.
	This implies strong existence, pathwise uniqueness, and uniqueness in law. 
	For a complete proof in a much more general setting, see \cite[Ch. IV, Section 9]{MR637061}.
\end{proof}
Similarly, since $\text{Var}(M^\gamma_T) = \E \left[ \Lambda^2 ~|~ \cF^N_T \right] - \E^2 \left[ \Lambda~|~\cF^N_T \right]$ is $\cF^N_T$-measurable, by Doob's lemma, there exists a deterministic and measurable function $V: D([0, T]) \rightarrow \R$ such that:
\[ \P(\d \omega)~a.s., \quad \text{Var}(M^\gamma_T) = \E[\Lambda^2 ~|~ \cF^N_T] - \E^2[ \Lambda~|~\cF^N_T ] = V(N). \]
\begin{remark}
	In Corollary~\ref{cor:aposteriori-dist} below, we show that the function $V$ admits an explicit expression.
\end{remark}
Thus, our optimization problem \eqref{eq:original-value} can be expressed as:
\begin{align*} & \inf_{\Psi \in \tilde{A}} \E \left[  \sum_{k \geq 0} \frac{1}{2}\int_0^T { \Psi^2_{k}(N_{\cdot \wedge \tau_k})(s) \indica{(\tau_k, \tau_{k+1}]}(s) \d s} + \kappa V(N) \right] \\
	& =  \inf_{\Psi \in \tilde{A}} \int \left[  \sum_{k \geq 0} \frac{1}{2}\int_0^T { \Psi^2_{k}(\tilde{n}_{\cdot \wedge \tilde{\tau}_k})(s) \indica{(\tilde{\tau}_k, \tilde{\tau}_{k+1}]}(s) \d s} + \kappa V(\tilde{n}) \right] Q_{\Psi, \mu}(\d \tilde{n}). 
\end{align*}
Here, $\tilde{\tau}_k$ denotes the jump times of $\tilde{n}$. 
The key observation is that this expression is entirely independent of the original probability space.
\subsection{Girsanov's theorem and the posterior distribution}
\label{sec:girsanov-reformulation}
Having established that our optimization problem is independent of the initial probability space, we propose a specific and convenient construction of the relevant objects using Girsanov's theorem.

Consider a probability space $(\Omega, \cF, (\cF_t), \Q)$ such that $(N_t)$ is a standard $(\cF_t)$-Poisson process with intensity $1$.
Let $(\tau_k)$ denote the successive jump times of $(N_t)$, and let $\cF^N_t = \sigma( N_r, r \leq t)$ for the natural filtration of $N$.
Let $\tilde{\cA}$ denote the set of all $\cF^N$-predictable processes $\gamma$ such that
\[ a.s., \quad  \int_0^T \gamma^2_u \d u < \infty. \]
Given $\gamma \in \tilde{\cA}$ and $y \in \R$, we consider $(Y_t)$ as the solution to the stochastic differential equation:
\[
Y_t = y + \int_0^t (b(Y_u) + \gamma_u) \d u - \int_0^t Y_{u-} \d N_u. 
\]
For all $\lambda \in \R_+$, we define:
\begin{align*}
	L_t(\lambda) &= \prod_{0 < \tau_k \leq t} f_\lambda(Y_{\tau_k-}) \exp \left(t - \int_0^t f_\lambda(Y_{u})  \d u \right).  
\end{align*}
\begin{lemma}
	Let $\tilde{N}_t = N_t-t$. The equation 
	\[ L_t = 1 + \int_0^t L_{u-} [f_\lambda(Y_{u-}) -1] \d \tilde{N}_u \]
	has a unique locally bounded solution ($\sup_{u \in [0, t]} |L_u| < \infty~a.s.$), given by $L_t(\lambda)$ as defined above. Furthermore, $t \mapsto L_t(\lambda)$ is a $(\Q, (\cF^N_t))$-local martingale. 
\end{lemma}
\begin{proof}
	This follows from \cite[Th. A.T4]{BP}.
\end{proof}
Now, consider $\Lambda$ an $\cF_0$-measurable random variable with probability distribution $\mu$. Since $(N_t)$ is a standard $(\cF_t)$ Poisson process under $\Q$, $\Lambda$ and $(N_t)_{t \geq 0}$ are independent under $\Q$.  
\begin{lemma}
	Under Assumption~\ref{ass:initial-ass}, $\E_\Q L_T(\Lambda) = 1$, and thus $(L_t(\Lambda))$ is a $(\Q, (\cF_t))$-martingale. 
\end{lemma}
\begin{proof}
	By assumption, for all $\lambda \in \R_+$, 
	\[ C(\lambda) := \sup_{x \geq 0} f_\lambda(x) < \infty. \]
	We first show that $\E_\Q L_T(\lambda) = 1$ for all $\lambda \in \R_+$.
	Since $L_{t} \leq C(\lambda)^{N_t} e^{t}$ and $\abs{f_{\lambda}(Y_{t-}) - 1} \leq C(\lambda) + 1$, it follows that:
	\[ \E_\Q \int_0^T L_{u-}(\lambda) \abs{f_\lambda(Y_{u-}) - 1} \d u \leq \int_0^T e^{u C(\lambda)}(1+C(\lambda))\d u < \infty.  \]
	By \cite[T8, p. 27]{BP}, this implies $\E_\Q L_T(\lambda) = 1$.
	Next, by Doob's lemma, there exists a measurable function $\Phi$ such that:
	\[ L_T(\lambda) = \Phi(N_{\cdot \wedge T}, \lambda). \]
	Since $(N_t)_{t \geq 0}$ and $\Lambda$ are independent under $\Q$, we have: 
	\[ \E_\Q [\Phi(N_{\cdot \wedge T}, \Lambda) ~|~ \Lambda] = \Psi(\Lambda), \]
	where $\Psi(\lambda) := \E_\Q \Phi(N_{\cdot \wedge T}, \lambda) = \E_\Q L_T(\lambda) = 1$.
	Thus: 
		\[
			\E_\Q L_T(\Lambda) = \E_\Q \Psi(\Lambda) = 1.
		\]
\end{proof}
By Girsanov's theorem, we obtain the following:
\begin{proposition}
	\label{prop:girsanov}
	On $\d \P := L_T(\Lambda) \d \Q$, $(N_t)$ is a point process with stochastic intensity $f_\Lambda(Y_{t-})$. For any measurable test function $\phi: \R \rightarrow \R_+$, it holds that:
\begin{align*}  \E_\P [ \phi (\Lambda ) ~|~ \cF^N_t ]  &= \frac{ \E_\Q[ \phi(\Lambda) L_t(\Lambda) ~|~ \cF^N_t] }{\E_\Q[ L_t(\Lambda) ~|~ \cF^N_t]} \\
		&= \frac{ \int \varphi(\lambda) L_t(\lambda) \mu(\d \lambda) }{ \int L_t(\lambda)  \mu(\d \lambda) } =: \frac{ \langle \phi L_t, \mu \rangle }{ \langle L_t, \mu \rangle }. 
\end{align*}
\end{proposition}
\begin{proof}
	This follows from \cite[Th. 3]{BP} and \cite[Lem. 5]{BP}. For completeness, we provide the arguments for the second part of the result. Let $A \in \cF^N_t$.  By definition of conditional expectation, we have $\E_\P[ \indica{A} \phi(\Lambda)] = \E_\P[ \indica{A} \E_\P [\phi(\Lambda) ~|~\cF^N_t]]$. Thus, 
	\[ \E_\Q [ \indica{A} \phi(\Lambda) L_T(\Lambda)] = \E_\Q[ \indica{A} L_T(\Lambda) \E_\P [\phi(\Lambda) ~|~\cF^N_t]]. \]
	Taking the conditional expectation with respect to $\cF^N_t$, we obtain:
	\[  \E_\Q [ \indica{A} \E_\Q[ \phi(\Lambda) L_T(\Lambda)~|~\cF^N_t] ] = \E_\Q[ \indica{A} \E_\Q[L_T(\Lambda)~|~\cF^N_t] \E_\P [\phi(\Lambda) ~|~\cF^N_t]]. \]
	Since this holds for all $A \in \cF^N_t$ it follows that:
	\[  \E_\Q[ \phi(\Lambda) L_T(\Lambda)~|~\cF^N_t] =   \E_\Q[L_T(\Lambda)~|~\cF^N_t] \E_\P [\phi(\Lambda) ~|~\cF^N_t], \]
	yielding the first equality for $\langle \phi, M^\gamma_t \rangle$. For the second equality, since $(N_t)$ and $\Lambda$ are independent under $\Q$, we have:
	\[ \E_\Q[ \phi(\Lambda) L_T(\Lambda)~|~\cF^N_t] = \int_{\R_+} \phi(\lambda) L_t(\lambda) \mu(\d \lambda).  \]
	This implies the stated formula.
\end{proof}
\begin{corollary}
	\label{cor:aposteriori-dist}
Thus, $a.s.$, the posterior distribution is given by:  
\[ M^\gamma_t := \cL_\P( \Lambda ~|~ \cF^N_t) \propto \prod_{\tau_k \leq t} f_\lambda(Y_{\tau_k-}) \exp \left( - \int_0^t f_\lambda(Y_s) \d s \right) \mu(\d \lambda). \] 
	 Additionally, the variance of the posterior distribution, $V(N) := Var(M^\gamma_T) = \E_\P[\Lambda^2~|~\cF^N_T] - \left( \E_\P(\Lambda~|~ \cF^N_T) \right)^2$, is given by:
	\[ V(N) =  \frac{ \langle \phi_2 L_t, \mu \rangle }{ \langle L_t, \mu \rangle } -  \left(\frac{ \langle \phi_1 L_t, \mu \rangle }{ \langle L_t, \mu \rangle }\right)^2,  \quad \text{ where } \quad \phi_1(\lambda) = \lambda \text{ and }\phi_2(\lambda) = \lambda^2. \]
\end{corollary}
\begin{remark}
Let $\phi: \R \rightarrow \R_+$ be a non-negative measurable function.
Then, $t \mapsto \langle \phi, M^\gamma_t \rangle = \E_\P [ \phi(\Lambda)~|~ \cF^N_t ]$ is a martingale. In other words, $M^\gamma_t$ is a measure-valued martingale. In particular, it follows that
$\E_\P \E_\P [\Lambda^2 ~|~ \cF^N _t] = \E \Lambda^2 < \infty$. 
Thus, for any admissible control $\gamma$, we have:
\[ \P(\d \omega) ~a.s., \quad \E_\P [\Lambda^2 ~|~ \cF^N_t] < \infty. \]
Furthermore, by the law of total variance:
\[ Var(\mu) = Var(M^\gamma_0) = \E_\P Var(M^\gamma_t) + Var( \E_\P[ \Lambda ~|~ \cF^N_t] ). \]
This implies that, in expectation, the variance of $M^\gamma_t$ decreases over time.
\end{remark}
Finally, we have the equivalence between the original control problem and the reformulation using Girsanov's theorem.
\begin{proposition}
	Let $\gamma \in \tilde{\cA}$. 
	\begin{enumerate}
		\item There exists $\Psi = (\Psi_k)_{k \geq 0}$ a collection of measurable functions from $D([0, T])$ to $L^2([0, T])$ such that $\gamma$ admits the representation
			\[ \gamma_t = \sum_{k \geq 0} \Psi_k(N_{\cdot \wedge \tau_k})(t) \indica{(\tau_k, \tau_{k+1}]}(t).   \]
		\item In addition, the law of $(N_t)_{t \in [0, T]}$ on $\d \P = L_T(\Lambda) \d \Q$ is equal to $Q_{\Psi, \mu}$.
	\end{enumerate}
	Therefore, the original problem control~\eqref{eq:original-value} and our reformulation using Girsanov's theorem are equivalent, namely:
	\begin{equation}  \inf_{\gamma \in \tilde{\cA} } \E_\Q \left[ L_T(\Lambda) \left( \frac{1}{2} \int_0^T \gamma^2_u \d u + \kappa V(N) \right) \right] = \inf_{\gamma \in \cA} \E_\P \left[  \frac{1}{2}\int_0^T { \gamma^2_u \d u} + \kappa \text{Var}(M^{\gamma}_T) \right].  
	\end{equation}
\end{proposition}
\begin{proof}
	As $\gamma \in \tilde{\cA}$ is $(\cF^N_t)$-predictable and $\Q(\int_0^T \gamma^2_s \d s < \infty) = 1$, Proposition~\ref{prop:caracterisation-progressible-processes} and the subsequent remark apply. This gives the first point. In addition, by Proposition~\ref{prop:girsanov}, on $\d \P = L_T(\Lambda) \d \Q$, $(N_t - \int_0^t f_\Lambda(Y_s) \d s)$ is a $(\P, \cF_t)$-martingale. We now apply \cite[Lem. 4]{zbMATH00971853} (see also~\cite[p. 469-478]{zbMATH03644252} for the result in a more general setting): on an enlarged probability space $(\Omega', \cF', (\cF'_t), \P')$, there exists a $(\cF'_t)$-Poisson measure $\Pi(\d s, \d z)$ on $\R_+ \times \R_+$ of intensity the Lebesgue measure $\d s \d z$ such that:
	\[  N_t = \int_0^t \int_{\R_+} \indic{z \leq f_\Lambda(Y_{s-})} \Pi(\d s, \d z), \quad t \geq 0. \]
As $\Lambda$ is $\cF'_0$-measurable, we deduce that $\Lambda$ and $\Pi$ are independent.
Therefore, $(\Lambda, Y_t, N_t, \gamma_t)$ satisfies the assumptions of Lemma~\ref{lem:triple-L-N-Y}. Consequently, the law of $(N_t)$ is characterized and:
\[ \cL\left((N_t)_{t \in [0, T]}\right) = Q_{\Psi, \mu}. \]
Altogether, it holds that:
\begin{align*}  &\E_\Q \left[ L_T(\Lambda) \left(  \frac{1}{2} \int_0^T \gamma^2_u \d u + \kappa V(N) \right) \right] \\
	\quad &=  \int \left[  \sum_{k \geq 0} \frac{1}{2} \int_0^T {\Psi^2_{k}(\tilde{n}_{\cdot \wedge \tilde{\tau}_k})(s) \indica{(\tilde{\tau}_k, \tilde{\tau}_{k+1}]}(s) \d s} + \kappa V(\tilde{n}) \right] Q_{\Psi, \mu}(\d \tilde{n}).
\end{align*}
This concludes the proof.
\end{proof}
\subsection{The value function}

\label{sec:the-value-function}
 We now generalize the formulation of the problem to define the value function.
\subsubsection*{Probability spaces.}

We consider the canonical space $(\Omega, \cF, (\cF^N_t), \Q)$ associated with a standard Poisson process $(N_t)$ with intensity $1$. Specifically, $\Omega$ is the space of right-continuous simple counting functions, where $\Omega \ni \omega = (\omega_t)_{t \geq 0}$ satisfies $\omega_0 = 0$ and is a piecewise constant function with a non-accumulating infinite sequence of jumps of size one. We define $N_t(\omega) = \omega_t$ for all $t \geq 0$. The space $\Omega$ is equipped with the Skorokhod topology, making it a Polish space. In addition, $\cF$ is the associated Borelian $\sigma$-algebra, which is equal to $\sigma(N_t, t \geq 0)$, and we define $\cF^N_t = \sigma(N_s, s \leq t)$. The measure $\Q$ is law of a standard Poisson process. We denote by $(\tau_k(\omega))$ the successive jump times of $(N_t(\omega))$, with the convention $\tau_0(\omega) = 0$.

\subsubsection*{Controls.}

Following \cite{BouchardTouzi} (or \cite{zbMATH06624241} in a more general setting) for $t \geq 0$, we define the translated filtration $(\cF^{N, t}_s)_{s \geq 0}$  as:
\[ \cF^{N, t}_s := \sigma( N_r - N_t, r \in [t, t \vee s]). \]
This filtration satisfies the property that, for all $t, s \in \R_+$, a random variable measurable with respect to $\cF^{N, t}_s$ is independent of $\cF^N_t = \cF^{N, 0}_t$.
We denote by $\cA_t$ the set of all $(\cF^{N,t}_s)$-predictable processes $\gamma$ such that, almost surely, $\int_0^T \gamma^2_u \d u < \infty$.  

\subsubsection*{Value function.}

Given $\gamma \in \cA_t, y \in \R$, and $\mu \in \cP_2(\R_+)$, we define:
\begin{align*} J^\gamma(t, y, \mu) := \E_\Q \left[ \int_{\R_+} \left( \int_t^T \frac{\gamma_s^2}{2} \d s + \kappa\text{Var}(M^{t, y, \mu, \gamma}_T) \right) L^{y, \gamma}_{t, T}(\lambda) \mu(\d \lambda) \right], 
\end{align*}
where, for all $s \in [t, T]$: 
\begin{align}
& L^{y, \gamma}_{t, s}(\lambda) = \prod_{ \tau_k \in  (t, s]} f_{\lambda}(Y^{t, y, \gamma}_{\tau_k-}) \exp \left( (s-t) -\int_t^s f_\lambda(Y^{t, y, \gamma}_{u}) \d u  \right), \label{eq:def-L-indices}\\
& Y^{t, y, \gamma}_{s} = y + \int_t^s \gamma_u \d u - \int_t^s Y^{t, y, \gamma}_{u-} \d N_u, \label{eq:def-Y-indices}
\end{align}
and the posterior measure $M^{t, y, \mu, \gamma}_s$ is defined by:
\begin{align}
& \langle \phi, M^{t, y, \mu, \gamma}_s \rangle  := \frac{  \langle \phi L^{y, \gamma}_{t, s}, \mu \rangle }{ \langle L^{y, \gamma}_{t, s}, \mu \rangle }, \label{eq:def-M-indices}
 \end{align}
for any non-negative measurable function $\phi: \mathbb{R}_+ \to \mathbb{R}_+$. The variance of $M^{t, y, \mu, \gamma}_T$ is given by:
\begin{align*}
& \text{Var}(M^{t, y, \mu, \gamma}_T) = \langle \phi_2, M^{t, y, \mu, \gamma}_T \rangle - (\langle \phi_1, M^{t, y, \mu, \gamma}_T \rangle)^2, \quad \text{ with } \phi_1(\lambda) = \lambda, \phi_2(\lambda) = \lambda^2. \end{align*}
The value function is defined as:
\[ v(t, y, \mu) := \inf_{\gamma \in \cA_t} J^\gamma(t, y, \mu). \]
We have shown in the previous sections that:
\begin{theorem}
	The value defined by the optimization problem~\eqref{eq:original-value} is equal to $v(0, y, \mu)$.
\end{theorem}

\subsection{Reduction to finite dimension}
\label{sec:reduction-finite-dim}
In what follows, we assume the following:
\begin{assumption}
There is a globally Lipschitz and bounded function $g: \R \rightarrow \R_+$, such that:
\[ \forall y \in \R, \lambda \in \R_+, \quad f_\lambda(y) = \lambda g(y). \]
\end{assumption}
Under this assumption, the optimization problem over the space of probability measures can be reduced to a finite-dimensional problem.
To proceed, given $\mu \in \cP_2(\R_+)$, we define for all $z \geq 0$ and $n \in \N$:
\[ \tilde{v}(t, y, z, n) := v(t, y, m_\mu(n, z)), \quad \text{ where } \quad m_\mu(n, z)(\d \lambda) := \frac{ \lambda^n e^{-\lambda z} \mu(\d \lambda) }{ \int_{\R_+}  \theta^n e^{-\theta z} \mu(\d \theta) }. \]
Note that $m_\mu(0,0) = \mu$. We also define:
\[ \Phi_\mu(n, z) := \int_{\R_+} \lambda^n e^{-\lambda z} \mu(\d \lambda), \]
and observe that:
\[ \forall k \in \N, \quad \int_{\R_+} \lambda^k m_\mu(n, z) (\d \lambda) = \frac{\Phi_\mu(n+k, z)}{\Phi_\mu(n, z)}. \] 
Furthermore, we define:
\begin{equation} 
	\label{eq:Psi_mu}
	\Psi_\mu(n, z) := \text{Var}(m_\mu(n, z)) = \frac{\Phi_\mu(n+2, z)}{ \Phi_\mu(n, z)} - \left( \frac{\Phi_\mu(n+1, z)}{ \Phi_\mu(n, z)} \right)^2. 
\end{equation}
Recall the definition of $M^{t, y, \mu, \gamma}_s$ and $Y^{t, y, \gamma}_s$ given by equations~\eqref{eq:def-M-indices} and ~\eqref{eq:def-Y-indices}. We introduce:
\begin{equation}
Z^{t, y, z, \gamma}_s := z + \int_t^s g(Y^{t, y, \gamma}_u) \d u. \label{eq:def-Z-indices}
\end{equation}
\begin{lemma}
	Let $\gamma \in \cA_t$. It holds a.s. that:
\[ M^{t, y, m_\mu(n, z), \gamma}_{s} = m_\mu(n + N_s - N_t, Z^{t, y, z, \gamma}_s), \]
\end{lemma}
\begin{proof}
This follows from the fact that, for every non-negative measurable function $\phi$, we have:
\begin{equation} 
	\label{eq:bayes-m-mu}
	\langle \phi, m_\mu(n + N_s - N_t, Z^{t, y, z, \gamma}_s) \rangle = \frac{\langle \phi L^{y, \gamma}_{t, s}, m_\mu(n, z) \rangle }{ \langle L^{y, \gamma}_{t, s}, m_\mu(n, z) \rangle }. 
\end{equation}
\end{proof}
As a corollary, we obtain:
\begin{proposition}
	It holds that:
	\[ \tilde{v}(t, y, z, n) = \inf_{\gamma \in \cA_t} \E_\Q \left[ \int_{\R_+} \left( \int_t^T \frac{\gamma_s^2}{2} \d s + \Psi_\mu(n + N_T - N_t, Z^{t, y, z, \gamma}_T) \right) L^{t, y, \gamma}_{T}(\lambda) m_\mu(n, z)(\d \lambda) \right]. \]
\end{proposition}
The key observation is that this optimization problem is now a classical finite-dimensional stochastic optimization problem. Henceforth, we focus our analysis on this finite-dimensional problem. To simplify notation, we write $v(t, y, z, n) = \tilde{v}(t, y, z, n)$.
\subsection*{Some properties of the posterior distribution}
We conclude this section by establishing some properties of $m_\mu(n, z)$.
\begin{lemma}\label{lemme:reg_post}
	Let $\Xi_\mu(n,z) := \Phi_\mu(n+1, z) / \Phi_\mu(n,z)$ denote the expected value of the posterior distribution. For all $n \in \N$ and $z \geq 0$, the following hold:
	\begin{enumerate}
		\item $\partial_z \Phi_\mu(n,z) = - \Phi_\mu(n+1, z)$.
		\item $\partial_z \Xi_\mu(n, z) = - \Psi_\mu(n,z)$.
		\item $\Xi_\mu(n+1, z) - \Xi_\mu(n, z) = \frac{\Psi_\mu(n,z)}{\Xi_\mu(n,z)}$.
		\item $\partial_z \Psi_\mu(n, z) = \Xi_\mu(n, z)\left[ \Psi_\mu(n, z) - \Psi_\mu(n+1, z) - \left( \frac{\Psi_\mu(n, z)}{\Xi_\mu(n, z)} \right)^2  \right]$.
	\end{enumerate}
\end{lemma}
\begin{remark}
	 Points 2 and 3 above indicate that the mean of the posterior distribution decreases between jumps but increases immediately after a jump. 
	However, the variance of the posterior distribution does not exhibit a similar behavior. For example, consider the case where $\mu(\d \lambda) = \frac{1}{2} \delta_0 + \frac{1}{2} \delta_{\lambda_{\max}}$. After one jump, the posterior distribution becomes $\delta_{\lambda_{\max}}$, with zero variance: for $n \geq 1, \Psi_\mu(n, z) = 0$. In this case, the variance drops to zero immediately after the first jump.
\end{remark}
\begin{lemma}
	Assume that $\text{Supp}(\mu) \subset [0, \lambda_{\max}]$. Then:
	\[
\Psi_\mu(n, z) \leq \frac{\lambda^2_{\max}}{4} \ \text{ and } \ 
	\abs{ \partial_z \Psi_\mu(n, z) } \leq \frac{1}{2} \lambda^3_{\max}. \]
\end{lemma}

\section{Dynamic programming principle and regularity}
\label{sec:PPD-reg}
We now establish that the value function is regular and satisfies a dynamic programming principle. Recall that we work in the canonical space defined in Section~\ref{sec:the-value-function}. For simplicity, given $\mu \in \cP_2(\R_+)$ and a measurable function $h: \R_+ \times \Omega \rightarrow \R_+$, we write:
\[ \E_\Q [ h(\Lambda, N)] := \int_{\R_+} \E_\Q [h(\lambda, N)] \mu(\d \lambda). \]
\subsection{Space regularity}
\label{sec:space-reg}
\label{sec:regulairty-value-function}
We now study the regularity of the value function with respect to $y$ and $z$. Recall that
\[ v(t, y, z, n) = \inf_{\gamma \in \cA_t } J_1^\gamma(t, y, z, n) + J^\gamma_2(t, y, z, n), \]
with
\begin{align*}
	J_1^\gamma(t, y, z, n) &:=  \E_\Q \left[ L^{\gamma, y}_{t, T}(\Lambda)  \int_t^T \frac{\gamma^2_s}{2} \d s \right],\\
	J_2^\gamma(t, y, z, n) &:=  \E_\Q \left[ L^{\gamma, y}_{t, T}(\Lambda) \Psi_\mu(n+N_T-N_t, z + \int_t^T g(Y^{t, y, \gamma}_s) \d s)  \right],
\end{align*}
and  $\Lambda$ is a random variable with distribution $\P(\Lambda \in \d \lambda) = \lambda^n e^{-\lambda z} \mu(\d \lambda) / \Phi_\mu(n, z)$.
\begin{assumption}
	\label{ass:standing-assumptions}
	We assume the following:
	\begin{enumerate}
		\item The prior distribution is compactly supported: there exists $\lambda_{\max} > 0$ such that:
			\[ \text{Supp}(\mu) \subset [0, \lambda_{\max}]. \]
		\item The function $g \in C^1(\R; \R_+)$ satisfies $\norm{g}_\infty + \norm{\frac{g'}{g}}_\infty < \infty$.
		\item The function $b \in C^1(\R)$ satisfies $\norm{b'}_\infty < \infty$.
	\end{enumerate}
\end{assumption}
The main result of this section is the following:
\begin{proposition}
	\label{prop:regularity}
	Under Assumption~\ref{ass:standing-assumptions}, there exists a constant $C_T$ such that, for all $n \geq 0$, $y, y' \in \R$, and $z, z' \geq 0$:
	\[ \abs{v(t, y, z, n) - v(t, y', z', n)} \leq C_T( |y-y'| + |z-z'|). \]
\end{proposition}
We now provide the proof of this result. To simplify notation, we assume $t = 0$ and write:
\[ L^{y, \gamma}_T(\lambda) = L^{y, \gamma}_{t, T}(\lambda). \]
\subsubsection{Regularity with respect to \texorpdfstring{$y$}{y}}
\begin{lemma}
	 There exists a constant $C_T$ such that, for all $\gamma \in \cA$:
	\[ \forall y \in \R, \quad \abs{\partial_y L^{y, \gamma}_T(\lambda)} \leq C_T L^{y, \gamma}_{T}(\lambda).  \]
\end{lemma}
\begin{proof}
	We prove the result with $C_T := \left(\norm{\frac{g'}{g}}_\infty + \lambda_{\max} T \norm{g'}_\infty\right) e^{T \norm{b'}_\infty}$.
	Observe that, immediately after a jump, $Y$ is reset to zero, so the initial condition$y$ is forgotten.
	Let $\varphi_s(y)$ denote the solution to the ordinary differential equation $\frac{d}{d s} \varphi_s(y) = b(\varphi_s(y)) + \gamma_u $ with initial condition $\varphi_0(y) = y$.

	\textbf{Case 1: $N_T = 0$.} If there are no jumps, then:
	\[ \partial_y L^{y, \gamma}_T(\lambda) = - \lambda L^{y, \gamma}_T(\lambda) \int_0^T g'(\varphi_s(y)) \partial_y \varphi_s(y) \d s. \]
	Since $\partial_y \varphi_s(y) = \exp\left(\int_0^s b'(\varphi_u(y)) \d u \right) \leq \exp(T \norm{b'}_\infty)$, the result holds.

	\textbf{Case 2: $N_T > 0$.} Let $\tau_1$ denote the time of the first jump of $N$. Then:
	\[ \partial_y L^{y, \gamma}_T(\lambda) = \left[ \frac{g'(\varphi_{\tau_1-}(y)) \partial_y \varphi_{\tau_1-}(y)  }{ g(\varphi_{\tau_1-}(y))} - \lambda \int_0^{\tau_1} g'(\varphi_s(y)) \partial_y \varphi_s(y) \d s \right] L^{y, \gamma}_T(\lambda).  \]
	The result follows from the assumptions on $g$ and $b$. 
\end{proof}
\begin{lemma}
	There exists a constant $C_T$ such that, for all $\gamma \in \cA$:
\[ \forall y, y' \in \R, \quad  L^{y', \gamma}_T(\lambda) \leq e^{C_T |y-y'|} L^{y, \gamma}_T(\lambda). \]
\end{lemma}
\begin{proof}
	Assume first that $y' > y$. The result follows by applying Grönwall's lemma. If now $y' < y$, define $\phi(s) = L^{y-s, \gamma}_T(\lambda)$. Then: 
	\[
\phi'(s) = - \partial_y L^{y-s, \gamma}_T(\lambda) \leq C_T \phi(s).
\]
By Grönwall's lemma, $\phi(s) \leq e^{C_T s} \phi(0)$. Choosing $s = y-y'$ completes the proof. 
\end{proof}
\begin{lemma}
	There is another constant $C_T$ such that, for all $\gamma \in \cA$, and for all $y, y'$ with $|y-y'| \leq 1$:
	\[ \abs{L^{y, \gamma}_T(\lambda) - L^{y', \gamma}_T(\lambda)} \leq C_T |y-y'| (L^{y, \gamma}_T(\lambda) + L^{y', \gamma}_T(\lambda)). \]
\end{lemma}
\begin{proof} Without loss of generality, assume $y < y' < y+1$. Then:
	\[
L^{y', \gamma}_T(\lambda) - L^{y, \gamma}_T(\lambda) \leq (e^{C_T(y'-y)} - 1) L^{y, \gamma}_T(\lambda) \leq C_Te^{C_T} |y-y'| L^{y, \gamma}_T(\lambda).\]
Similarly:
\[
L^{y, \gamma}_T(\lambda) - L^{y', \gamma}_T(\lambda) \leq C_T e^{C_T} |y-y'| L^{y', \gamma}_T(\lambda).
\]
Combining these, the result follows.
\end{proof}
\begin{corollary}
	Since $L^{y', \gamma}_T(\lambda) \leq e^{C_T} L^{y, \gamma}_T(\lambda)$, there exists another constant
 $C_T$ such that, for all $y, y'$ with $|y-y'| \leq 1$,
	\[ \abs{L^{y, \gamma}_T(\lambda) - L^{y', \gamma}_T(\lambda)} \leq C_T |y-y'| L^{y, \gamma}_T(\lambda). \]
\end{corollary}
\begin{lemma}
	There exists a constant $C_T$ such that, for all $\gamma \in \cA$, and for all $y, y'$ with $|y-y'| \leq 1$: 
	\begin{align*} \abs{J^\gamma_1(t, y', z, n) - J_1^\gamma(t, y, z, n)} &\leq C_T |y-y'| J^\gamma_1(t, y, z, n), \\
	\abs{J_2^\gamma(t, y', z, n) - J_2^\gamma(t, y, z, n)} &\leq C_T |y-y'|(1+J^\gamma_2(t, y, z, n)). \end{align*}
\end{lemma}
\begin{proof}
	For the first inequality:
	\begin{align*}  \abs{J^\gamma_1(t, y', z, n) - J^\gamma_1(t, y, z, n)} & \leq \E_\Q \left[ \abs{L^{y', \gamma}_{t, T}(\Lambda) - L^{y, \gamma}_{t, T}(\Lambda)} \int_t^T \frac{\gamma_s^2}{2} \d s \right] \\
		& \leq C_T |y-y'| \E_\Q \left[ L^{y, \gamma}_{t, T}(\Lambda) \int_t^T \frac{\gamma_s^2}{2} \d s \right] = C_T |y-y'| J^\gamma_1(t, y, z, n). 
	\end{align*}
	For the second inequality: 
	\begin{align*}
	& \abs{J^\gamma_2(t, y, z, n) - J^\gamma_2(t, y', z, n)} \leq \\
	& \quad \E_\Q \abs{L^{\gamma, y}_T(\Lambda) - L^{\gamma, y'}_T(\Lambda)} \Psi_\mu(n+N_T-N_t, Z^{t, y, z, \gamma}_T) \\ 
	&  \quad + \E_\Q \left[ L^{\gamma, y'}_T(\Lambda) \left( \Psi_\mu(n+N_T-N_t, Z^{t, y, z, \gamma}_T) - \Psi_\mu(n+N_T-N_t, Z^{t, y', z, \gamma}_T) \right)  \right]  =: A+B.
	\end{align*}
	The term $A$ satisfies:
	\[
		A \leq C_T |y-y'| \E_\Q L^{\gamma, y}_{t, T}(\Lambda) \Psi_\mu(n+N_T-N_t, Z^{t, y, z, \gamma}_T) = C_T|y-y'| J^\gamma_2(t, y, z, n).
\]
 The term $B$ satisfies:
\begin{align*}  B & \leq \E_\Q  L^{y', \gamma}_{t, T}(\Lambda) \norm{\partial_z \Psi_\mu(n+N_T-N_t, \cdot)}_\infty \int_t^T \norm{g'}_\infty e^{\norm{b'}_\infty T} |y-y'| \d s \\
	& \leq C_T \E_\Q L^{y, \gamma}_{t, T}(\Lambda)  |y-y'|  = C_T |y-y'|.
\end{align*}
	Here, we used that $\sup_{n \in \N} \norm{\partial_z \Psi_\mu(n, \cdot)} < \infty$. Combining these, the result follows.
\end{proof}
Combining these results, we deduce that:
\[ \abs{J^\gamma(t, y, z, n) - J^\gamma(t, y', z, n)} \leq C_T (1+J^\gamma(t, y, z, n)) |y-y'|. \]
Let $y,y' \in \R$ be fixed. Consider $\epsilon > 0$ and an $\epsilon$-optimal control $\gamma$ such that:
\[ v(t, y, z, n) \geq J^\gamma(t, y, z, n) - \epsilon. \]
Then:
\begin{align*}
	v(t, y', z, n) - v(t, y, z, n) & \leq J^\gamma(t, y', z, n) - J^\gamma(t, y, z, n) + \epsilon \\
	& \leq C_T |y-y'| (1+J^\gamma(t, t, z, n)) + \epsilon \\
	& \leq C_T |y-y'| (1 + \epsilon + v(t, y, z, n)) + \epsilon. 
\end{align*}
Letting $\epsilon \downarrow 0$, we obtain:
\[ v(t, y', z, n) - v(t, y, z, n) \leq C_T |y'-y| (1+v(t, y, z, n)). \]
Additionally, note that:
\[ v(t, y, z, n) \leq J^{0}(t, y, z, n) \leq \frac{\lambda^2_{\max}}{4} < \infty.\]
By exchanging the roles of $y$ and $y'$, we conclude:
\[ \abs{  v(t, y, z, n) - v(t, y', z, n) } \leq C_T |y-y'|. \] \qed

\subsubsection{Regularity with respect to \texorpdfstring{$z$}{z}}
The regularity with respect to $z$ is established using similar arguments. Recall that:
\[ \Phi_\mu(n, z) = \int_{\R_+} \lambda^n e^{-\lambda z} \mu(\d \lambda), \quad m_\mu(n, z)  =  \frac{\lambda^n e^{-\lambda z}}{ \Phi_\mu(n, z) }\mu(\d \lambda). \]
\begin{lemma}
	There exist constants $C, \eta > 0$ such that, for all $n \in \N$:
	\begin{enumerate}
		\item For all $z, z' \geq 0$, $\abs{\Phi_\mu(n, z) -\Phi_\mu(n, z')} \leq C |z-z'| \Phi_\mu(n, z)$.
		\item For all $z, z' \geq 0$, $\abs{m_\mu(n,z) - m_\mu(n, z')} \leq C |z-z'| [ m_\mu(n, z) + m_\mu(n, z')]$.
		\item For all $z, z' \geq 0$ with $|z-z'| < \eta$, $\abs{m_\mu(n,z) - m_\mu(n, z')} \leq C |z-z'| m_\mu(n, z)$.
	\end{enumerate}
\end{lemma}
\begin{proof}
	For the first point:
	\[ \abs{\Phi_\mu(n, z) - \Phi_\mu(n, z')} \leq \int_{\R_+} \lambda^n e^{-\lambda z} \abs{1 - e^{-\lambda (z' - z)}} \mu(\d \lambda). \]
Since $Supp(\mu) \subset [0, \lambda_{\max}]$, the result follows. For the second point:
\begin{align*}
	\abs{ m_\mu(n, z) - m_\mu(n, z')} \leq  \left[  \frac{ \abs{\Phi_\mu(n, z') - \Phi_\mu(n, z)} e^{-\lambda z}}{ \Phi_\mu(n, z) \Phi_\mu(n, z') } + \frac{e^{-\lambda z'}(e^{-\lambda (z-z')}  - 1)}{ \Phi_{\mu}(n, z')} \right] \lambda^n \mu(\d \lambda).
\end{align*}
Using the first point, we deduce that the second inequality holds. For the third point, we verify that, for $|z-z'| \leq \eta$, with $\eta$ sufficiently small:
\[ m_\mu(n, z') \leq C m_\mu(n, z). \]
 Indeed:
\[ m_\mu(n, z') = \frac{ \lambda^n e^{-\lambda z} e^{-\lambda (z'-z)} \mu(\d \lambda)}{ \Phi_\mu(n, z') - \Phi_\mu(n, z) + \Phi_\mu(n, z)}. \]
Since $\abs{\Phi_\mu(n, z') - \Phi_\mu(n, z)} \leq \frac{1}{2} \Phi_\mu(n, z)$ for $\abs{z-z'}$ sufficiently small, the result follows.
\end{proof}
Since under $\Q$, the Poisson process $N$ is independent of $\Lambda$, we have:
\[ J^\gamma_1(t, y, z, n) = \int_{\R_+ } \left[ \E_\Q L^{y, \gamma}_{t, T}(\lambda) \int_t^T \frac{\gamma^2_s}{2} \d s \right]  m_\mu(n, z) (\d \lambda).  \]
From the third point of the previous lemma, for all $z, z' \geq 0$ with $\abs{z-z'} \leq \eta$, it follows that:
\[ \abs{ J^\gamma_1(t, y, z, n) - J^\gamma_1(t, y, z', n) } \leq  C_T \abs{z-z'} J^\gamma_1(t, y, z, n).  \]
Similarly:
\[ \abs{J^\gamma_2(t, y, z, n) - J^\gamma_2(t, y, z', n) } \leq C_T |z-z'|(1 + J^\gamma_2(t, y, z, n)). \]
Thus, as previously, we deduce:
\[ \abs{v(t, y, z, n)- v(t, y, z', n)} \leq C_T \abs{z-z'}. \]
This completes the proof of Proposition~\ref{prop:regularity}. 
\subsection{A first dynamic programming principle}
\label{sec:PPD-deterministic-times}
Recall that the value function is defined by:
\begin{align*}
 v(t, y, z, n) &:= \inf_{\gamma \in \cA_t} \E_\Q \langle L^{y, \gamma}_{t, T}, m_\mu(n, z) \rangle  \left[ \int_t^T \frac{\gamma^2_u}{2} \d u + \kappa\Psi_\mu(n+N_T-N_t, Z_{T}^{t, y, z, \gamma}) \right] \\
	&= \inf_{\gamma \in \cA_t} J^\gamma(t,y,z,n).
\end{align*}
where $\Psi_{\mu}$ is defined in equation~$\eqref{eq:Psi_mu}$. In this section, we establish the following initial version of the dynamic programming principle, which we will later generalize to include $(\cF^N)$-stopping times. 
\begin{proposition}\label{prop:ppd}
	For all $s \in [t, T]$, it holds that:
	\begin{align*} v(t, y, z, n) = \inf_{\gamma \in \cA_t} & \E_\Q \langle L^{y, \gamma}_{t, s}, m_\mu(n, z)\rangle \left[ \int_t^s \frac{\gamma^2_u}{2} \d u + v(s, Y^{t, y, \gamma}_s, Z^{t, y, z, \gamma}_s, n + N_s - N_t)  \right]. 
	\end{align*}
\end{proposition}
We now provide the proof of this result. Recall that for $\gamma \in \cA_t$, 
\begin{align*}
	J^\gamma_1(t, y, z, n) &= \E_\Q \left[ \langle L^{\gamma, y}_{t, T}, m_\mu(n, z) \rangle \int_t^T \frac{\gamma^2_u}{2} \d u \right], \\
	J^\gamma_2(t, y, z, n) &= \E_\Q \left[ \langle L^{\gamma, y}_{t, T}, m_\mu(n, z) \rangle \Psi_\mu(n+N_T-N_t, Z^{t, y, z, \gamma}_T ) \right].
\end{align*}

We start with the following key lemma. Given $\gamma \in \cA_t$, $\mathrm{w} \in \Omega$, and $s \geq t$, we set
\[ \gamma^{s, \mathrm{w}} := \omega \mapsto \gamma( \mathrm{w} \otimes_s \omega), \]
where for all $u \geq 0$:
\[  (\mathrm{w} \otimes_s \omega)(u) = \begin{cases} \mathrm{w}(u) \quad \text{ if } u \leq s \\ \mathrm{w}(s) + \omega(u) - \omega(s) \quad \text{ if } u > s. \end{cases} \]
\begin{lemma}
	\label{lem:pseudo-markov-ppd}
	Let $\gamma \in \cA_t$ and let $S$ be a $\cF^N$-stopping time such that $S \in [t, T]$. Then, $\Q(\d \omega)$ a.s.:
	\begin{align*}
	& \E_\Q \left[ \langle L^{y, \gamma}_{t, T}, m_\mu(n, z) \rangle  \int_{S}^T \frac{\gamma^2_u}{2} \d u  \mid\cF^N_S \right](\omega) \\
	& \quad=    \langle L^{y, \gamma^{S(\omega), \omega}}_{t, S(\omega)}, m_\mu(n, z) \rangle J^{\gamma^{S(\omega), \omega}}_1(S(\omega), Y^{t, y, \gamma^{S(\omega), \omega}}_{S(\omega)}, Z^{t, y, z, \gamma^{S(\omega), \omega}}_{S(\omega)}, n + N_{S(\omega)} - N_t),
	\end{align*}
	and:
	\begin{align*}
	& \E_\Q [ \langle L^{y, \gamma}_{t, T}, m_\mu(n, z) \rangle \Psi_\mu(n+N_T-N_t, Z^{t, y, z, \gamma}_T ) ~|~\cF^N_S ](\omega)  \\
	& \quad = \langle L^{y, \gamma^{S(\omega), \omega}}_{t, S(\omega)}, m_\mu(n, z) \rangle J^{\gamma^{S(\omega), \omega}}_2(S(\omega), Y^{t, y, \gamma^{S(\omega), \omega}}_{S(\omega)}, Z^{t, y, z, {\gamma^{S(\omega), \omega}}}_{S(\omega)}, n + N_{S(\omega)} - N_t).
	\end{align*}
\end{lemma}
\begin{proof}
	First, note that for all $\lambda \geq 0$:
	\[  L^{y, \gamma}_{t,T}(\lambda) = L^{y, \gamma}_{t, S}(\lambda) L^{Y^{t, y, \gamma}_{S}, \gamma}_{S, T}(\lambda).  \]
	Combined with \eqref{eq:bayes-m-mu}, this implies:
	\[ \langle L^{y, \gamma}_{t, S},  m_\mu(n, z) \rangle \langle L^{Y^{t, y, \gamma}_{S}, \gamma}_{S, T}, m_\mu(n + N_S - N_t, Z^{t, y, z, \gamma}_S)  \rangle = \langle L^{y, \gamma}_{t, T}, m_\mu(n, z) \rangle.  \]
	We now prove the second equality. Let $A \in \cF^N_S$.  To simplify notation, define:
	\[ y^\gamma_S := Y^{t, y, \gamma}_S, \quad z^\gamma_S := Z^{t, y, z, \gamma}_S, \quad n_S := n + N_S - N_t. \]
	Additionally, set:
	\[ \Delta Z^\gamma_{S, T}  := Z_{T}^{S, y^\gamma_S, z^\gamma_S, \gamma} - z^\gamma_S = \int_S^T g(Y^{t, y, \gamma}_{u}) \d u. \]
Then:
	\begin{align*} & \E_\Q  \langle L^{y, \gamma}_{t, T}, m_\mu(n, z) \rangle \indica{A} \Psi_\mu(n  + N_T - N_t, Z^{t, y, z, \gamma}_T)  \\
		& \quad = \E_\Q \langle L^{y, \gamma}_{t, S},  m_\mu(n, z) \rangle \langle L^{y^\gamma_S, \gamma}_{S, T}, m_\mu(n_S, z^\gamma_S)   \rangle  \indica{A} \Psi_\mu(n_S + N_T - N_S, z^\gamma_S + \Delta Z^\gamma_{S, T} ) \\
		& \quad =  \int  \langle L^{y, \gamma}_{t, S}(\omega),  m_\mu(n, z) \rangle \indica{A}(\omega)   \\
		& \qquad ~~~~~  \E_\Q \left[   \langle L^{y^\gamma_S, \gamma}_{S, T}, m_\mu(n_S, z^\gamma_S)   \rangle  \Psi_\mu(n_S  + N_T - N_S, z^\gamma_S + \Delta Z^\gamma_{S, T})   ~|~ \cF^N_S \right](\omega)   \Q(\d \omega) \\
		& \quad = \int  \langle L^{y, \gamma}_{t, S}(\omega),  m_\mu(n, z) \rangle \indica{A}(\omega)  J^{\gamma^{S(\omega), \omega}}_2(S, y^{\gamma^{S(\omega), \omega}}_{S}, z^{\gamma^{S(\omega), \omega}}_{S}, n_{S})(\omega) \Q(\d \omega).
	\end{align*}
	The last equality follows from the strong Markov property of the standard Poisson process. Since this holds for any $A \in \cF^N_S$, the result is established.
\end{proof}
We now provide the proof of Proposition~\ref{prop:ppd}, which follows from standard arguments. Recall that: 	\[
	J^\gamma(t, y, z, n) = J^\gamma_1(t, y, z, n) + \kappa J^\gamma_2(t, y, z, n).
\]
Using Lemma~\ref{lem:pseudo-markov-ppd} with a deterministic time $S = s \in [t, T]$, we have:
\begin{align*}
	v(t, y, z, n) = \inf_{\gamma \in \cA} \E_\Q \left[ \langle L^{y, \gamma}_{t, s}, m_\mu(n, z) \rangle \left\{ \int_t^s \frac{\gamma^2_u}{2} \d u  + J^\gamma(s, Y^{t, y, \gamma}_s, Z^{t, y, z, \gamma}_s, n + N_s - N_t)    \right\} \right].
\end{align*}
It follows that:
\[ v(t, y, z, n) \geq \inf_{\gamma \in \cA} \E_\Q \left[ \langle L^{y, \gamma}_{t, s}, m_\mu(n, z) \rangle \left\{ \int_t^s \frac{\gamma^2_u}{2} \d u  + v(s, Y^{t, y, \gamma}_s, Z^{t, y, z, \gamma}_s, n + N_s - N_t)    \right\} \right]. \]
For the reverse inequality, fix $\epsilon > 0$ and $s \in (t, T)$. Using the regularity of $J^\gamma$ and $v$ established in Section~\ref{sec:regulairty-value-function}, there exists $\alpha > 0$ such that, for all $n \in \mathbb{N}$, there is a partition $(B_{i}^{n})_{i \in \mathbb{N}}$ of $\mathbb{R} \times \mathbb{R}_{+}$ with centers $(y_{i}^{n}, z_{i}^{n}) \in B_{i}^{n}$ such that, for all $(n, i) \in \mathbb{N}^{2}$ and $(y, z) \in B_{i}^{n}$:
	\[
		|y - y_{i}^{n}| + |z - z_{i}^{n}| < \alpha \quad \text{and} \quad |v(s, y, z, n) - v(s, y_{i}^{n}, z_{i}^{n}, n)| < \epsilon.
	\]
For each $(y_{i}^{n}, z_{i}^{n}, n)$, there exists an $\epsilon$-optimal control $\alpha(y_{i}^{n}, z_{i}^{n}, n) \in \cA$ such that:
\[ J^{\alpha(y_{i}^{n}, z_{i}^{n}, n)}(s, y^{n}_i, z^{n}_i, n) \leq v(s, y^{n}_i, z^{n}_i, n) + \epsilon. \]
Define the control $\gamma \in \cA$ by:
\[
\gamma^*_u = \begin{cases} \gamma_u \quad  \text{ if } \quad u < s \\ \sum_{i\geq0} \alpha_u^{i, n_{s}} \indica{B^{n_{s}}_{i}}(Y^{t, y, \gamma}_s, Z^{t, y, z, \gamma}_s) \quad \text{ if } \quad u \in [s, T]. \end{cases}
\]
where $\alpha^{i, n} := \alpha(y_{i}^{n}, z_{i}^{n}, n)$ for $i, n \geq 0$ and $n_{s} := n + N_s - N_t$.
To simplify notation, let $y_s = Y^{t, y, \gamma}_s$ and $z_s = Z^{t, y, z, \gamma}_s$. Then:
\begin{align*} J^{\gamma^*}(s, y_s, z_s, n_s) &=  \sum_{i} \left[J^{\gamma^*}(s, y_s, z_s, n_s) - J^{\gamma^*}(s, y^{n_s}_{i}, z^{n_s}_{i}, n_s)\right] \indica{B^{n_s}_{i}}(y_s, z_s) \\
	& \quad +  \sum_{i} \left[ J^{\gamma^*}(s, y^{n_s}_{i}, z^{n_s}_{i}, n_s)- v(s, y^{n_s}_{i}, z^{n_s}_{i}, n_s)\right] \indica{B^{n_s}_{i}}(y_s, z_s) \\
	& \quad + \sum_{i} \left[ v(s, y^{n_s}_{i}, z^{n_s}_{i}, n_s)- v(s, y_s, z_s, n_s)\right] \indica{B^{n_s}_{i}}(y_s, z_s) \\
	& \quad + v(s, y_s, z_s, n_s) \\
& \leq  3\epsilon + v(s, y_s, z_s, n_s). 
\end{align*}  
Thus:
\begin{align*}
	v(t, y, z, n) \leq 3 \epsilon + \E_\Q \left[ \langle L^{y, \gamma}_{t, s}, m_\mu(n, z) \rangle \left\{ \int_t^s \frac{\gamma^2_u}{2} \d u  + v(s, y_s, z_s, n_s) \right\} \right].
\end{align*}
Letting $\epsilon \rightarrow 0$ and noting that $\gamma$ is arbitrary between $t$ and $s$, we deduce:
\[ 	v(t, y, z, n) \leq \inf_{\gamma \in \cA} \E_\Q \left[ \langle L^{y, \gamma}_{t, s}, m_\mu(n, z) \rangle \left\{ \int_t^s \frac{\gamma^2_u}{2} \d u  + v(s, y_s, z_s, n_s) \right\} \right]. \]
This completes the proof of Proposition~\ref{prop:ppd}. \qed

\subsection{Time regularity}
\label{sec:time-regularity}
We now establish the following:
\begin{proposition}
	\label{prop:time-continuity}
	Under Assumption~\ref{ass:standing-assumptions}, the value function $v$ is continuous.
\end{proposition}
Given Proposition~\ref{prop:regularity}, it suffices to prove that the value function is continuous with respect to time. Since the prior distribution is compactly supported in $[0, \lambda_{\max}]$, we recall that:
\[ v(t, y, z, n) \leq \frac{\lambda^2_{\max}}{4}. \]
We begin with a preliminary lemma. For $\gamma \in \cA$, let $\varphi^{\gamma}_{t, s}(y)$ denote the solution to the ordinary differential equation:
\[  \frac{\d }{\d s}\varphi^\gamma_{t, s}(y) = b(\varphi^\gamma_{t, s}(y)) + \gamma_s \quad \text{ with } \quad \varphi^\gamma_{t,t}(y) = y. \] 
\begin{lemma}
	There exists a constant $C_T > 0$ such that, for all $t \leq s \leq T$: 
	\[ \forall y \in \R, \quad \abs{\varphi^\gamma_{t, s}(y) -y} \leq C_T \int_t^s \abs{\gamma_u} \d u. \]
\end{lemma}
\begin{proof}
 The result follows by applying Grönwall's lemma to $s \mapsto \varphi^\gamma_{t, s}(y) - y$, using the Lipschitz property of $b$.
\end{proof}
Assume without loss of generality that $s > t$ and $s-t \leq 1$. To simplify notation, define: 
	\[
Y_s := Y^{t, y, \gamma}_s, \quad n_s := n + N_{s} - N_t, \quad z_s := z + \int_t^s g(Y^{t, y, \gamma}_{u} \d u).
\]
Using Proposition~\ref{prop:ppd}, we write:
\begin{align}
	v(t, y, z, n) - v(s, y, z, n) & =  \inf_{\gamma \in \cA}\E_\Q \langle L^{y, \gamma}_{t, s}, m_\mu(n, z) \rangle \left( \int_t^s \frac{\gamma^2_u}{2} \d u \right. \nonumber  \\
														    & \quad \quad + [v(s, Y_s, z_s, n_s) - v(s, y, z, n)] \indic{N_{s} = N_t}  \nonumber \\
														    & \quad \quad \left. +  [v(s, Y_s, z_s, n_s) - v(s, y, z, n)] \indic{N_s > N_t} \vphantom{\frac12}\right) \nonumber \\
														    &	=: \inf_{\gamma \in \cA} \left[ A^\gamma_1+A^\gamma_2+A^\gamma_3 \right]. \label{eq:vt - vs}
\end{align}
Since $v$ is bounded by $\frac{\lambda_{\max}^2}{4}$, and $\E_\Q \langle L^{y, \gamma}_{t, s}, m_\mu(n, z) \rangle  \indic{N_{s} > N_t} \leq \P (\tilde{N}_{s} > \tilde{N}_t) \leq C (s-t)$, where under $\P$, $(\tilde{N}_t)$ is a Poisson process with intensity $\lambda_{\max} \norm{g}_\infty$, it follows that $|A^\gamma_3| \leq C (s-t)$ for any control $\gamma \in \cA$.
For the second term, using the regularity results from Section~\ref{sec:regulairty-value-function}:
\begin{align*}
	\abs{v(s, Y_s, z_s, n_s) - v(s, y, z, n)} \indic{N_s = N_t} & \leq C_T \left[ \abs{\varphi^\gamma_{t, s}(y) - y} + \abs{z_s-z} \right] \\
														       & \leq C_T( \int_t^{s} \abs{\gamma_u} \d u + (s-t) \norm{g}_\infty).
\end{align*}
Since $\abs{\gamma_u} \leq 1 + \gamma^2_u$, we deduce that, for any $\gamma \in \cA$:
\begin{align*} |A^\gamma_1| + |A^\gamma_2|  + |A^\gamma_3|& \leq C_T (s-t) + C_T \E_\Q  \langle L^{y, \gamma}_{t, s}, m_\mu(n, z) \rangle  \int_t^s \frac{\gamma^2_u}{2} \d u. \end{align*}
Define the subclass of admissible controls:
\[ \cA_0(t, y, z, n) := \{ \gamma \in \cA: J^\gamma(t, y, z, n) \leq J^0(t, y, z, n) \}. \]
Clearly, we can restrict to this class since:
\[ v(t, y, z, n) = \inf_{\gamma \in \cA_0(t, y, z, n)} J^\gamma(t, y, z, n). \]
Fix $\epsilon > 0$. Let $\gamma \in \cA_0(t, y, z, n)$ be an $\epsilon$-optimal for \eqref{eq:vt - vs}. Then: 
\[ \abs{v(t, y, z, n) - v(s, y, z, n)} \leq \epsilon + C_T(s-t) + C_T \E_\Q  \langle L^{y, \gamma}_{t, s}, m_\mu(n, z) \rangle  \int_t^s \frac{\gamma^2_u}{2} \d u. \]
By the dominated convergence theorem, the right-hand side is less than $3 \epsilon$ when $s$ is sufficiently close to $t$.
Thus, the value function is continuous in time. \qed

\subsection{The dynamic programming principle}
\label{sec:PPD}
We now establish the general dynamic programming principle, which extends to stopping times.
\begin{proposition}
	\label{prop:PPD}
	Under Assumption~\ref{ass:standing-assumptions}, let $\tau$ be an $(\cF^N_t)$ stopping time such that, a.s. $\tau \in [t, T]$. Then:
	\begin{align*} v(t, y, z, n) = \inf_{\gamma \in \cA_t} & \E_\Q \langle L^{y, \gamma}_{t, \tau}, m_\mu(n, z)\rangle \left[ \int_t^\tau \frac{\gamma^2_u}{2} \d u + v(\tau, Y^{t, y, \gamma}_\tau, Z^{t, y, z, \gamma}_\tau, n + N_\tau - N_t) \right]. 
	\end{align*}
\end{proposition}
\begin{proof}
	Given the continuity of the value function (Proposition~\ref{prop:time-continuity}), the proof follows similarly to that of Proposition~\ref{prop:ppd}, with a slight modification: the partition is now defined on $[0, T] \times \R \times \R_+$, to account for stopping times. Using Lemma~\ref{lem:pseudo-markov-ppd}, the proof is completed as in the deterministic case.
\end{proof}

\section{Viscosity solution properties and characterization}
\label{sec:visc}
In this section, to simplify notation, we let $x = (y, z) \in \R \times \R_+$. With this convention, the value function is denoted by $v(t, x, n)$, where $t \in [0, T]$ and $n \in \N$. Based on the results of the previous section, we expect that the value function formally satisfies:
\begin{align}
	\label{eq:HJB-reformulated}
	-\partial_t v + H(x, \nabla_x v) + \theta_n(x) [v - v(t, \psi_n(x))] = 0,
\end{align}
with the terminal condition:
\begin{equation}  v(T, x, n) = \kappa \Psi_\mu(n, z). \label{eq:terminal-condition-HJB}
\end{equation}
Here:
\[ H(x, p) := \frac{1}{2} p_1^2 - B(x) \cdot p, \quad B(x) := \begin{pmatrix} b(y) \\ g(y) \end{pmatrix}, \quad \psi_n(x) := (0, z, n+1) \]
and:
\[ \theta_n(x) := \frac{\Phi_\mu(n+1, z)}{\Phi_\mu(n, z)} g(y).  \]
We denote $D := [0, T] \times \R \times \R_+$ and let $BUC(D)$ denote the space of bounded and uniformly continuous functions on $D$.
Following \cite{MR861089}, we consider the following notion of viscosity solutions. 
\begin{definition}
	\label{def:viscosity-sol}
	A function $u: D \times \N \rightarrow \R$ is a viscosity subsolution of \eqref{eq:HJB-reformulated} if, for all $n \in \N$, the mapping $(t, x) \mapsto u(t, x, n)\in BUC(D)$, and for all $(t_\circ, x_\circ, n_\circ) \in D \times \N$, for all functions $\varphi \in C(D \times \N)$ that are $C^1$ in an open neighborhood of $(t_\circ, x_\circ, n_\circ)$ and satisfy $\max(u-\varphi) = (u-\varphi)(t_\circ, x_\circ, n_\circ) = 0$, it holds that if $t_\circ < T$:
	\[
	-\partial_t \varphi(t_\circ,x_\circ, n_\circ) + H(x_\circ, \partial_x \varphi(t_\circ, x_\circ, n_\circ)) + \theta_{n_\circ}(x_\circ) \left[ u(t_\circ, x_\circ, n_\circ) - u(t_\circ, \psi_{n_\circ}(x_\circ)) \right] \leq 0.
	\]
	A function $w: D \times \N \rightarrow \R$ is a viscosity supersolution of \eqref{eq:HJB-reformulated} if, for all $n \in \N$, the mapping $(t, x) \mapsto w(t, x, n)\in BUC(D)$, and for all $(t_\circ, x_\circ, n_\circ) \in D \times \N$, for all functions $\varphi \in C(D \times \N)$ that are $C^1$ in an open neighborhood of $(t_\circ, x_\circ, n_\circ)$ and satisfy $\min(w-\varphi) = (w-\varphi)(t_\circ, x_\circ, n_\circ) = 0$, it holds that if $t_\circ < T$:
	\[
	-\partial_t \varphi(t_\circ,x_\circ, n_\circ) + H(x_\circ, \partial_x \varphi(t_\circ, x_\circ, n_\circ)) + \theta_{n_\circ}(x_\circ) \left[ w(t_\circ, x_\circ, n_\circ) - w(t_\circ, \psi_{n_\circ}(x_\circ)) \right] \geq 0. 
	\]
	Finally, a function $u: D \times \N \rightarrow \R$ is a viscosity solution of \eqref{eq:HJB-reformulated}if it is both a viscosity subsolution and a viscosity supersolution of this equation.
	\end{definition}

We shall use an equivalent reformulation of viscosity solutions. Following \cite{MR861089}, we have:
	\begin{lemma}
		\label{lem:reformulation-visc-solution}
		A function $u:D\times\N \mapsto \R$ with $(t, x) \mapsto u(t, x, n) \in BUC(D)$ for all $n \in \mathbb{N}$ is a viscosity subsolution of \eqref{eq:HJB-reformulated} if and only if, for all $(t_\circ, x_\circ, n_\circ) \in D \times \N$,  and for all functions $\varphi \in C(D \times \N)$ that are $C^1$ in an open neighborhood of $(t_\circ, x_\circ, n_\circ)$ and satisfy $\max(u-\varphi) = (u-\varphi)(t_\circ, x_\circ, n_\circ) = 0$, it holds that if $t_0 < T$:
	\[
	-\partial_t \varphi(t_\circ,x_\circ, n_\circ) + H(x_\circ, \partial_x \varphi(t_\circ, x_\circ, n_\circ)) + \theta_{n_\circ}(x_\circ) \left[ \varphi(t_\circ, x_\circ, n_\circ) - \varphi(t_\circ, \psi_{n_\circ}(x_\circ)) \right] \leq 0. 
	\]
	A similar statement holds for viscosity supersolutions.
	\end{lemma}
	The only difference from Definition~\ref{def:viscosity-sol} is in the zero-order term, where the subsolution $u$ is replaced by the test function $\varphi$. For the proof, see \cite[Lem. 2.1]{MR861089}.

\bigbreak

	The main result of this section is the following.
	\begin{theorem}
		Under Assumption~\ref{ass:standing-assumptions}, the value function $v$ is the unique viscosity solution of \eqref{eq:HJB-reformulated} that satisfies the terminal condition~\eqref{eq:terminal-condition-HJB}. 
	\end{theorem}
	We prove successively that $v$ is a viscosity subsolution and a viscosity supersolution of \eqref{eq:HJB-reformulated}, and that a comparison theorem holds. We first establish the following crucial lemma.
\begin{lemma}\label{lem:Nmartingale}
	The process
		\[
			\widetilde{N}_{t} := N_t - \int_{0}^{t}\E_{\P}\left[\Lambda \mid \mathcal{F}_{u}^{N}\right] g(Y_u^{t, y, \gamma})\d u, \quad t \geq 0
		\]
\end{lemma}
is an $\mathcal{F}^{N}-$martingale under $\mathbb{P}$.

\begin{proof}
The process is integrable (since $\Lambda \in L^{1}(\mathbb{P})$) and $\mathcal{F}^{N}-$ adapted. Moreover: \begin{align*}
	\mathbb{E}\left[N_{t} - N_{s} \mid \mathcal{F}_{s}^{N} \right] &= \mathbb{E}\left[\mathbb{E}\left[N_{t} - N_{s}  \mid \sigma(\Lambda, \mathcal{F}_{s}^{N}) \right] \mid \mathcal{F}_{s}^{N} \right] = \mathbb{E}\left[\int_{s}^{t}\Lambda g(Y_u^{t, y, \gamma}) \d u \mid \mathcal{F}_{s}^{N}\right] \\
&= \int_{s}^{t}\mathbb{E}\left[\Lambda g(Y_u^{t, y, \gamma}) \mid \mathcal{F}_{s}^{N}\right] \d u = \int_{s}^{t}\mathbb{E}\left[\mathbb{E}\left(\Lambda g(Y_u^{t, y, \gamma}) \mid \mathcal{F}_{u}^{N} \right) \mid \mathcal{F}_{s}^{N}\right] \d u \\
&= \mathbb{E}\left[\int_{s}^{t}\mathbb{E}\left(\Lambda  \mid \mathcal{F}_{u}^{N} \right)g(Y_u^{t, y, \gamma}) \d u \mid \mathcal{F}_{s}^{N}\right].
\end{align*}
Thus, we obtain the result.
\end{proof}

\subsection{Subsolution property}
\label{sec:sub-sol-visc}
	In this section we prove: 
	\begin{lemma}
		The value function $v$ is a viscosity subsolution of \eqref{eq:HJB-reformulated}.
	\end{lemma}
	\begin{proof}
		Let $x_\circ := (y_\circ, z_\circ)$, $(t_\circ, x_\circ, n_\circ) \in D \times \N$, and let $\varphi: D \times \N \rightarrow \R$ be a test function as in Lemma~\ref{lem:reformulation-visc-solution}.
		Consider a constant deterministic control $\gamma \in \R$. Define $n_{s} := n_\circ + N_{s} - N_{t_\circ}$ and $X_s^{t_{\circ}, x_{\circ}, \gamma} = (Y_s^{t_\circ, y_\circ, \gamma}, Z_s^{t_\circ, y_\circ, z_\circ, \gamma})$, for $t_{\circ} \leq s \leq T$. Let $T_1 := \inf\{ s \geq t_\circ, N_s \neq N_{s-} \}$ denote the time of the first jump.
		Since $\gamma$ is a valid control and $s \wedge T_1$ is a stopping time, the dynamic programming principle~\ref{prop:PPD} implies:
		\[ \E_\P \left[ v(t_\circ, x_\circ, n_\circ) - \int_{t_\circ}^{s \wedge T_1 } \frac{\gamma^2}{2} \d u - v(s \wedge T_1, X_{s \wedge T_1}^{s_{\circ}, x_{\circ}, \gamma}, n_{s \wedge T_1} )  \right]  \leq 0. \]
		Since $\varphi \geq v$, it follows that:
\[	\E_\P \left[ \varphi(t_\circ, x_\circ, n_\circ) - \int_{t_\circ}^{s \wedge T_1 } \frac{\gamma^2}{2} \d u - \varphi(s \wedge T_1, X_{s \wedge T_1}^{t_{\circ}, x_{\circ}, \gamma}, n_{s \wedge T_1})  \right]  \leq 0. \]
By It\^{o}'s formula:
\begin{align*}  \varphi(s \wedge T_1, X_{s \wedge T_1}^{t_{\circ}, x_{\circ}, \gamma}, n_{s \wedge T_1})) &= \varphi(t_\circ, x_\circ, n_\circ) \\
& \! + \int_{t_\circ}^{s \wedge T_1} [\partial_t \varphi + (\partial_y \varphi) (b + \gamma) + (\partial_z \varphi) g](u, X_u^{t_{\circ}, x_{\circ}, \gamma}, n_{u}) \d u  \\
& \! + \int_{t_\circ}^{s \wedge T_1} \left[\varphi(u, \psi_{n_u}(X_u^{t_{\circ}, x_{\circ}, \gamma})) - \varphi(u, X_{u-}^{t_{\circ}, x_{\circ}, \gamma}, n_u)\right] \theta_{n_u}(X_u^{t_{\circ}, x_{\circ}, \gamma}) \d u \\
&  \! + \int_{t_\circ}^{s \wedge T_1} \left[\varphi(u, \psi_{n_u}(X_u^{t_{\circ}, x_{\circ}, \gamma})) - \varphi(u, X_{u-}^{t_{\circ}, x_{\circ}, \gamma}, n_u)\right] \d \tilde{N}_u,
\end{align*}
where $\d \tilde{N}_u = \d N_u - \theta_{n_u}(X_u^{t_{\circ}, x_{\circ}, \gamma}) \d u  = \d N_u - \E [\Lambda ~|~ \cF^N_u] g(Y_u^{t_{\circ}, y_{\circ}, \gamma}) \d u$.
Taking the expectation with respect to $\E_\P$ and using Lemma~\ref{lem:Nmartingale}, we obtain:
\[
\frac{1}{s-t_\circ} \E_\P \left[ - \int_{t_\circ}^{s \wedge T_1 } \frac{\gamma^2}{2} \d u - \int_{t_\circ}^{s \wedge T_1} (\cL^\gamma \varphi)(u, X_u^{t_{\circ}, x_{\circ}, \gamma}, n_u) \d u \right]  \leq 0, \]
where the generator $\cL^\gamma$ acting on $\varphi$ is: 
\begin{equation} \label{eq:generator_with_t_and_gamma}
	(\cL^\gamma \varphi)(t,x,n) := [\partial_t \varphi + (\partial_y \varphi) (b + \gamma) + (\partial_z\varphi) g](t, x, n) + [\varphi(t, \psi_n(x)) - \varphi(t, x, n)] \theta_n(x). 
\end{equation}
Since $u < T_1$, we have $n_u = n_\circ$, and similarly, $X_u^{t_\circ, x_\circ, \gamma}$ is a deterministic function. Choose $s = t_\circ + \frac{1}{k}$, $k \in \N^*$, sufficiently large so that $(u, X_u^{t_\circ, x_\circ, \gamma})$ remains in the open neighborhood where $\varphi$ is $C^1$, for all $u \leq s\wedge T_1$. Applying the dominated convergence theorem and then the mean value theorem, we deduce:
\[ -\frac{\gamma^2}{2} - \cL^\gamma \varphi (t_\circ, x_\circ, n_\circ) \leq 0. \]
Optimizing over $\gamma$ by choosing $\gamma = -\partial_y \varphi(t_\circ, x_\circ, n_\circ)$, we obtain the inequality required by Lemma~\ref{lem:reformulation-visc-solution}.
\end{proof}

\subsection{Supersolution property}
\label{sec:supersol-visc}
	In this section we prove:
	\begin{lemma}
		The value function $v$ is a viscosity supersolution of \eqref{eq:HJB-reformulated}.
	\end{lemma}
	\begin{proof}
	We rely on Lemma~\ref{lem:reformulation-visc-solution}. Additionally, we may restrict to test functions $\varphi$ for which the minimum is strict, i.e., $(v-\varphi)(t_\circ, x_\circ, n_\circ) = 0$ and:
	\begin{equation}
		\label{eq:super1}
	(t, x, n) \neq (t_\circ, x_\circ, n_\circ) \implies (v-\varphi)(t, x, n) > 0.  
\end{equation}
	Consider $\varphi$ such a test function $\varphi$. Suppose, for contradiction, that:
\[	-\partial_t \varphi(t_\circ,x_\circ, n_\circ) + H(x_\circ, \partial_x \varphi(t_\circ, x_\circ, n_\circ)) + \theta_{n_\circ}(x_\circ) \left[ \varphi(t_\circ, x_\circ, n_\circ) - \varphi(t_\circ, \psi_{n_\circ}(x_\circ)) \right] < 0. 
\]
	\begin{enumerate}[label=\textbf{\arabic*}.]
	\item There exists $r > 0$ sufficiently small such that:
\begin{equation}  
		\label{eq:condict-super}
	-\partial_t \varphi + H(x, \partial_x \varphi) + \theta_{n}(x) \left[ \varphi- \varphi(t, \psi_{n}(x)) \right] < 0. 
\end{equation}
holds on $B(t_{\circ}, x_{\circ}, n_{\circ}; r)$ the ball centered at $(t_{\circ}, x_{\circ}, n_{\circ})$ with radius $r$. Define:
\[
	\widetilde{B}(t_{\circ}, x_{\circ}, n_{\circ}; r) := \left\{(t, \psi_n(x)) \ \text{ for } \ (t, x, n) \in B(t_{\circ}, x_{\circ}, n_{\circ}; r)\right\} \backslash B(t_{\circ}, x_{\circ}, n_{\circ}; r),
\]
 which is bounded. By (\ref{eq:super1}), there exists $\eta > 0$ such that:
	\begin{equation}\label{eq:etasuper2b}
		\eta := \min_{(\widetilde{B} \cup \partial B)(t_{\circ}, x_{\circ}, n_{\circ}; r)}(v-\varphi).
	\end{equation}
\item Let $\gamma$ be an arbitrary control, and let $\theta_\gamma$ be the first exit time of $B(t_{\circ}, x_{\circ}, n_{\circ}; r)$. By It\^o's formula:
\begin{align*}
	v(t_{\circ}, x_{\circ}, n_{\circ}) &= \varphi(t_{\circ}, x_{\circ}, n_{\circ}) \\
						      &= \varphi(\theta_\gamma, X_{\theta_\gamma}^{t_\circ, x_\circ, \gamma}, n_{\theta_\gamma}) - \int_{t_\circ}^{\theta_\gamma} \cL^{\gamma_u} \varphi(u, X_u^{t_\circ, x_\circ, \gamma}, n_u) \d u \\ 
						      & \quad - \int_{t_{\circ}}^{\theta_\gamma}\left[ \varphi(u, \psi_{n_u}(X_u^{t_\circ, x_\circ, \gamma})) - \varphi(u, X_u^{t_\circ, x_\circ, \gamma}, n_u) \right] \widetilde{N}(\d u),
\end{align*}
where $\widetilde{N}$ is the martingale process introduced in Lemma~\ref{lem:Nmartingale}. 
From \eqref{eq:condict-super} we have:
\[ \sup_{\gamma \in \R} \left[ -\cL^\gamma \varphi - \frac{1}{2} \gamma^2 \right] < 0. \]
Therefore:
\begin{align*}
	v(t_{\circ}, x_{\circ}, n_{\circ}) & \leq \mathbb{E}\left[\varphi \left(\theta_\gamma, X_{\theta_\gamma}^{t_\circ, x_\circ, \gamma}, n_{\theta_\gamma}\right) + \int_{t_\circ}^{\theta_\gamma} \frac{\gamma^2_u}{2}  \d u \right].
\end{align*}
	Using (\ref{eq:etasuper2b}), this gives:
\[
	v(t_{\circ}, x_{\circ}, n_{\circ}) \leq  -\eta + \mathbb{E}\left[v\left(\theta_\gamma, X_{\theta_\gamma}^{t_\circ, x_\circ, \gamma}, n_{\theta_\gamma}\right) + \int_{t_\circ}^{\theta_\gamma} \frac{\gamma_{u}^{2}}{2}\d u \right].
\]
\end{enumerate}
Since this inequality holds for any control $\gamma$, the latter inequality is in contradiction with Proposition~\ref{prop:PPD}.
\end{proof}

\subsection{Comparison theorem}
\label{sec:comparaison-visc}
We now establish a comparison principle for viscosity solutions of \eqref{eq:HJB-reformulated}. The main result of this section is:
\begin{proposition}
	Let $u$ be a viscosity subsolution and let $w$ be a viscosity supersolution of the PDE~\eqref{eq:HJB-reformulated} in the sense of Definition~\ref{def:viscosity-sol}. Assume that $u(T, x, n) \leq w(T, x, n)$ for all $(x, n) \in (\R \times \R_+) \times \N$. 
	Then:
	\[ u(t, x, n) \leq w(t, x, n), \quad \forall t, x, n \in D \times \N. \]
\end{proposition}
In particular, there exists at most one bounded and uniformly continuous viscosity solution of \eqref{eq:HJB-reformulated} that satisfies the terminal condition $v(T, x, n) = \kappa\Psi_n(x)$.
\begin{proof}
	Suppose, for contradiction, $\theta := \sup_{D \times \N } u-w > 0$.
	\begin{enumerate}[label=\textbf{\arabic*}.]
		\item 	Let $\epsilon, \delta, \beta, \alpha > 0$. Consider the auxiliary function:
			\begin{align*} 
				\Phi_n(t, x, t', x') &= u(t, x, n) - w(t', x', n) - \frac{1}{2 \epsilon} \norm{x-x'}^2 - \frac{1}{2 \epsilon} |t-t'|^2 + \beta (t' - T) \\
		 & \quad -  \delta F(\norm{x}) - \delta F(\norm{x'}) - \alpha n, \end{align*}
			where $\norm{x} := \sqrt{y^2 + z^2}$ for $x = (y, z)$ and:
\[ F(r) := \log (1+r^2), \quad r \geq 0. \]
For $\epsilon, \delta, \beta, \alpha$ sufficiently small, it holds that:
\begin{equation}
	\label{eq:contradic}
	\sup_{n \in \N} \sup_{ D^2} \Phi_n > \theta/2. 
\end{equation}
Since $u$ and $w$ are bounded and $\delta > 0, \alpha > 0$, the supremum is attained at some point $\bar{t}, \bar{x}, \bar{t}', \bar{x}', \bar{n}$.
\item From $\Phi_{\bar{n}}(\bar{t}, \bar{x}, \bar{t}', \bar{x}') \geq 0$ and the boundedness of $u$ and $w$, with $C = \norm{u}_\infty + \norm{w}_\infty$, we have:
	\[  \frac{1}{2 \epsilon} \norm{\bar{x} - \bar{x}'}^2 + \frac{1}{2 \epsilon} |\bar{t} - \bar{t}'|^2 + \delta F(\norm{\bar{x}}) + \delta F(\norm{\bar{x}'}) + \alpha \bar{n} \leq C.   \]
	From $\Phi_{\bar{n}}(\bar{t}, \bar{x}, \bar{t'}, \bar{x'}) \geq \Phi_{\bar{n}}(\bar{t'}, \bar{x'}, \bar{t'}, \bar{x'})$, it follows:
	\[ u(\bar{t}, \bar{x}, \bar{n}) - u(\bar{t'}, \bar{x'}, \bar{n}) \geq \frac{1}{2\epsilon} \norm{\bar{x} - \bar{x'}}^2 + \frac{1}{2 \epsilon} |\bar{t}-\bar{t'}|^2 + \delta F(\norm{\bar{x}}_2) - \delta F(\norm{\bar{x'}}_2). \]
For $\rho \geq 0$, define:
$D_\rho = \{ (t, x), (t', x') \in D^2:  |t-t'|^2 + \norm{x-x'}_2^2 \leq \rho\}$, 
and:
\[ m^\alpha_u(\rho) = 2 \sup\{ |u(t, x, n) - u(t', x', n)|, \quad (t, x), (t', x') \in D_\rho, n \in \N, n \leq C/\alpha \}. \]
Since $\norm{\bar{x}-\bar{x'}}^2 + |\bar{t}-\bar{t'}|^2 \leq 2C \epsilon$ and $\abs{\norm{\bar{x}} - \norm{\bar{x'}}} \leq \norm{\bar{x} - \bar{x'}}$, we deduce: 
\[  \frac{1}{\epsilon} \norm{\bar{x}-\bar{x'}}^2 + \frac{1}{\epsilon} |\bar{t} - \bar{t'}|^2 \leq m^\alpha_u(4C \epsilon) + \delta \norm{F'}_\infty \sqrt{8 C \epsilon}.  \]
As $(t, x) \mapsto u(t, x, n)$ is uniformly continuous and bounded, $m^\alpha_u: \R_+ \rightarrow \R_+$ is bounded, continuous, and $m^\alpha_u(0+) = m^\alpha_u(0) = 0$. A similar statement holds for $m^\alpha_w$. 
\item If $\bar{t} = T$, then:
	\begin{align*}
		\Phi_{\bar{n}}(\bar{t}, \bar{x}, \bar{t'}, \bar{x'}) & \leq u(\bar{t}, \bar{x}, \bar{n}) - w(\bar{t'}, \bar{x'}, \bar{n}) \\
								   & \leq w(\bar{t}, \bar{x}, \bar{n}) - w(\bar{t'}, \bar{x'}, \bar{n}) \\
		& \leq \frac{1}{2} m^\alpha_w(2 C \epsilon). 
	\end{align*}
\item Similarly, if $\bar{t'} = T$, then:
	\[ \Phi_{\bar{n}}(\bar{t}, \bar{x}, \bar{t'}, \bar{x'}) \leq \frac{1}{2} m^\alpha_u(2 C \epsilon). \] 
\item If $\bar{t} < T$ and $\bar{t'} < T$, consider: 
\[
\varphi_{u}(t, x) := \frac{1}{2 \epsilon} \norm{x-\bar{x'}}^2 + \frac{1}{2 \epsilon} |t-\bar{t'}|^2 + \delta F(\norm{x}).
\]
Since $(\bar{t}, \bar{x}) \in \text{argmax} (u- \varphi_{u})$, the viscosity subsolution property of $u$ implies:
	\[ - \frac{\bar{t} - \bar{t'}}{\epsilon} + H\left(\bar{x}, \frac{\bar{x} - \bar{x'}}{\epsilon} + \delta F'(\norm{\bar{x}}) \frac{\bar{x}}{\norm{\bar{x}}} \right) + \theta_{\bar{n}}(\bar{x}) [u(\bar{t}, \bar{x}, \bar{n}) - u(\bar{t}, \psi_{\bar{n}}(\bar{x}))]  \leq 0. \]
	Similarly, consider:
 \[
\varphi_{v}(t', x') = -\frac{1}{2\epsilon} \norm{\bar{x} - x'}^2 - \frac{1}{2\epsilon} |\bar{t}-t'|^2 - \delta F(\norm{x'}) + \beta(t' - T).
\]
Since $(\bar{t'}, \bar{x'}) \in \text{argmin}(w-\varphi_{v})$, the viscosity supersolution property of implies:
	\[ -\beta - \frac{\bar{t} - \bar{t'}}{\epsilon} + H\left( \bar{x'}, \frac{\bar{x} - \bar{x'}}{\epsilon} - \delta F'(\norm{\bar{x'}}) \frac{\bar{x'}}{\norm{\bar{x'}}} \right)  + \theta_{\bar{n}}(\bar{x'}) [w(\bar{t'}, \bar{x'}, \bar{n}) - w(\bar{t'}, \psi_{\bar{n}}(\bar{x'}))] \geq 0. \] 
Combining these inequalities, we obtain $\beta \leq \Delta_1 + \Delta_2+\Delta_3$ where: 
\begin{align*} \Delta_1 & :=  H\left( \bar{x'}, \frac{\bar{x} - \bar{x'}}{\epsilon} - \delta F'(\norm{\bar{x'}}) \frac{\bar{x'}}{\norm{\bar{x'}}}\right) -   H \left(\bar{x}, \frac{\bar{x} - \bar{x'}}{\epsilon} + \delta F'(\norm{\bar{x}}) \frac{\bar{x}}{\norm{\bar{x}}} \right) \\
		\Delta_2 & :=  [\theta_{\bar{n}}(\bar{x'}) - \theta_{\bar{n}}(\bar{x})] [w(\bar{t'}, \bar{x'}, \bar{n}) - w(\bar{t'}, \psi_{\bar{n}}(\bar{x'}))] \\
\Delta_3 & := \theta_{\bar{n}}(\bar{x})  [w(\bar{t'}, \bar{x'}, \bar{n}) - w(\bar{t'}, \psi_{\bar{n}}(\bar{x'}))  - u(\bar{t}, \bar{x}, \bar{n}) + u(\bar{t}, \psi_{\bar{n}}(\bar{x}))]. \end{align*}
For $\Delta_1$ using $H(x, p) = \frac{1}{2} p_1^2 - B(x) \cdot p$, we have:
	\begin{align*} \Delta_1 &= \underbrace{- \frac{\delta}{2} \left[ F'(\norm{\bar{x}'}) \frac{\bar{x}'}{\norm{\bar{x}'}} + F'(\norm{\bar{x}}) \frac{\bar{x}}{\norm{\bar{x}}} \right]_1 \left[ 2 \frac{\bar{x}-\bar{x}'}{\epsilon} + \delta F'(\norm{\bar{x}}) \frac{\bar{x}}{\norm{\bar{x}}} - \delta F'(\norm{\bar{x}'}) \frac{\bar{x}'}{\norm{\bar{x}'}} \right]_1}_{A_1}  \\
	& \quad + \underbrace{\left[ B(\bar{x}) - B(\bar{x'})\right] \cdot \frac{\bar{x}-\bar{x}'}{\epsilon} }_{A_2}  + \underbrace{\delta F'(\norm{\bar{x}}) \cdot \frac{B(\bar{x}) \cdot \bar{x}}{\norm{\bar{x}}} + \delta F'(\norm{\bar{x}'}) \frac{B(\bar{x}') \cdot \bar{x}' }{ \norm{\bar{x}'}} }_{A_3}.  
	\end{align*}
We used here the notation $[(y,z)]_1 = y$. We have:
\[ |A_1| \leq 2 \delta \left( \frac{\sqrt{2C}}{\sqrt{\epsilon}} + \delta\right), \quad |A_2| \leq \norm{\nabla B}_\infty  \left[ m^\alpha_u(4C \epsilon) + \delta \sqrt{8C \epsilon}\right]. \] 
In addition, there is a constant $\widetilde{C}$ such that $\abs{B(x) \cdot x} \leq \widetilde{C} (\norm{x}^2 + \norm{x})$. 
Using that  $\sup_{r \in \R_+} \abs{ (1+r) F'(r)} = 1+\sqrt{2}$, we deduce that:
\[ |A_3| \leq (1+\sqrt{2}) \widetilde{C} \delta. \] 
Denote by
\[ k(\alpha, \epsilon, \delta) :=  2 \delta \left( \frac{\sqrt{2C}}{\sqrt{\epsilon}} + \delta\right)  +  \norm{\nabla B}_\infty  \left[ m^\alpha_u(4C \epsilon) + \delta \sqrt{8C \epsilon}\right] + (1+\sqrt{2}) \widetilde{C} \delta, \]
	we retain that $\Delta_1 \leq k(\alpha, \epsilon, \delta)$.
	For $\Delta_2$, since $g \in C^{1}$ and by Lemma~\ref{lemme:reg_post}, $L := \sup_{n \in \N} \norm{\nabla \theta_n}_\infty < +\infty$, so:
	\[
|\Delta_2| \leq 2 \norm{w}_\infty L  \sqrt{2C \epsilon}.
\]
	For $\Delta_3$, let $\bar{x}^* = (0, \bar{x}_2)$ and $\bar{x}'^* = (0, \bar{x}'_2)$. Since: 
	\[  \Phi_{\bar{n}}(\bar{t}, \bar{x}, \bar{t}', \bar{x}') \geq \Phi_{\bar{n}+1}(\bar{t}, \bar{x}^*, \bar{t}', \bar{x}'^*), \]
	we deduce:
	\begin{align*} w(\bar{t'}, \bar{x'}, \bar{n}) - w(\bar{t'}, \psi_{\bar{n}}(\bar{x}'))  - u(\bar{t}, \bar{x}, \bar{n}) + u(\bar{t}, \psi_{\bar{n}}(\bar{x})) &\leq - \frac{1}{2\epsilon} \norm{\bar{x} - \bar{x}'}^2 + \frac{1}{2\epsilon} |\bar{x}_2 - \bar{x}'_2|^2  \\
	& \quad + \delta [F(|\bar{x}_2|) - F(\norm{\bar{x}})]\\
	& \quad +  \delta [F(|\bar{x}'_2|) - F(\norm{\bar{x}'})]  \\
	& \quad + \alpha ((n+1) - n)  \\
	& \leq \alpha,
	\end{align*}
	using that $F$ is non-decreasing and that $|\bar{x}_2| = \norm{\bar{x}^*} \leq \norm{\bar{x}}$, 
	since the $y$-coordinate is reset to zero after a jump.
\item For $\tilde{k}(\alpha, \epsilon, \delta) := k(\alpha, \epsilon, \delta) + 2 \norm{w}_\infty L \sqrt{2 C \epsilon} + \alpha$, we have $\beta \leq \tilde{k}(\alpha, \epsilon, \delta)$. Since:
	\[ \limsup_{\alpha \downarrow 0} \lim\sup_{\epsilon \downarrow 0} \limsup_{\delta \downarrow 0} \tilde{k}(\epsilon, \delta) = 0, \]
	there exists $\alpha(\beta) > 0$,  $\epsilon(\beta, \alpha) > 0$ and $\delta(\beta, \alpha, \epsilon)$ such that:
	\[ \forall \alpha \in (0, \alpha(\beta)), \forall \epsilon \in (0, \epsilon(\beta, \alpha)), \forall \delta \in (0, \delta(\beta, \alpha, \epsilon)), \quad \tilde{k}(\alpha, \epsilon, \delta) < \beta. \]
	This choice is in contradiction with $\beta \leq \tilde{k}(\alpha, \epsilon, \delta)$. Thus, either $\bar{t} = T$ or $\bar{t'} = T$, and:
	\[ \Phi_{\bar{n}}(\bar{t}, \bar{x}, \bar{t'}, \bar{x'}) \leq \frac{1}{2} (m^\alpha_w + m^\alpha_u)(2C \epsilon). \]
	For $\epsilon$ sufficiently small, this contradicts \eqref{eq:contradic}, completing the proof.
\end{enumerate}
\end{proof}
\section{Examples}
\label{sec:examples}
To illustrate the results numerically, we consider the drift $b(y) = -y$. The controlled process is:
\[	Y_t = y + \int_0^t (-Y_u + \gamma_u) \d u - \int_0^t Y_{u-} \d N_u, \quad t \in [0, T], y \in \R.
\]
We present two examples using different prior distributions and intensity functions. 
\subsection{First example}
We first consider an intensity function of the form:
\[
f_{\lambda}(y) := \lambda \exp(2(y-1)),
\]
and a prior distribution:
\[
	\mu := \sum_{i=1}^{k}p_{i}\delta_{\lambda_i},
\]
where $\sum_{i=1}^{k}p_i = 1$ and $k \geq 2$. To satisfy Assumption~\ref{ass:standing-assumptions}, the function $f_{\lambda}$ can be artificially bounded by a constant $C > 0$, which does not significantly affect the results. Notably, this family of priors is conjugate: the posterior distribution retains the same form with updated weights $(p_1, \ldots, p_k)$. Within our framework, the dimensionality of the problem is reduced to two, as opposed to the original $k-1$.

We solve the PDE \eqref{eq:HJB-reformulated} numerically using a standard explicit scheme. The convergence of this scheme to the unique viscosity solution can be analyzed following \cite[Thm. 2.1]{barles1991convergence}.

\begin{figure}[!ht]
\centering
\includegraphics[scale=0.65]{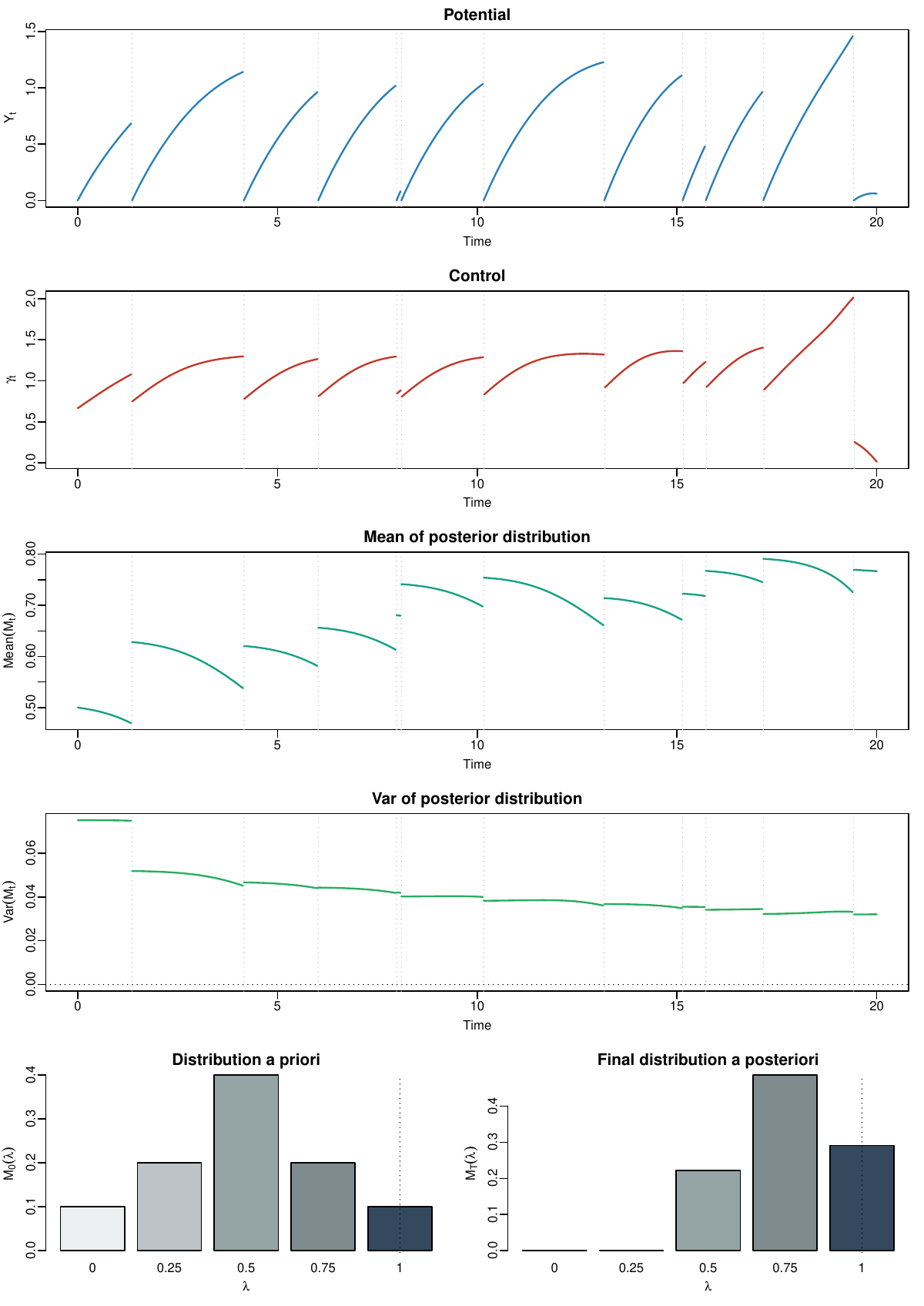}
\caption{An optimal trajectory with intensity $\lambda \exp(2(y-1))$ and prior distribution $\frac{1}{10}\left(\delta_{0} + 2\delta_{0.25} + 4\delta_{0.5} + 2\delta_{0.75} + \delta_{1}\right)$. The true value is $\lambda = 1$.}\label{x2_dirac}
\end{figure}

In Figure~\ref{x2_dirac}, we present a simulated path of the optimal strategy, where the initial measure is $m_{0} := \frac{1}{10}\left(\delta_{0} + 2\delta_{0.25} + 4\delta_{0.5} + 2\delta_{0.75} + \delta_{1}\right)$ and the true parameter value is $\Lambda = 1$. The top plot shows the evolution of the potential over time. The second plot depicts the optimal control, while the third and fourth plots display the mean and variance of the posterior distribution, respectively. The bottom-left plot shows $m_0$ as a histogram, and the bottom-right plot presents $m_T$.

Throughout the trajectory, the controller applies a control $t \mapsto \gamma_{t}$ that remains close to 1. Near the terminal time, after the final jump, the control decreases, as further increasing the potential is not cost-effective.

Although the prior is centered at $\lambda = 0.5$ with low probability assigned to $\lambda = 1$, the mean of the posterior distribution increases on average. By the end, the posterior assigns the highest probabilities to $\lambda = 0.75$ and $\lambda = 1$, excluding $\lambda = 0$ and nearly ruling out $\lambda = 0.25$. Notably, the posterior probability of $\lambda = 0$ becomes zero immediately after observing a jump.

\subsection{Second example}

For the second example, we consider the intensity function:
\[
	f_{\lambda}(y) := \lambda \left[\frac{1}{1 + e^{-100(y-1)}}\right].
\]
This function approximates the discontinuous function $(\lambda, y) \mapsto \lambda\mathbf{1}_{\{y \geq 1\}}$. For the discontinuous case, when $y < 1$, the controller decides whether to increase the potential to 1. Once the potential reaches 1, it remains fixed. The continuous version should exhibit similar behavior. We choose the uniform prior distribution on $[0, 2]$:
\[
	\mu := \mathcal{U}\left([0, 2]\right). \]
\begin{figure}[!ht]
\centering
\includegraphics[scale=0.65]{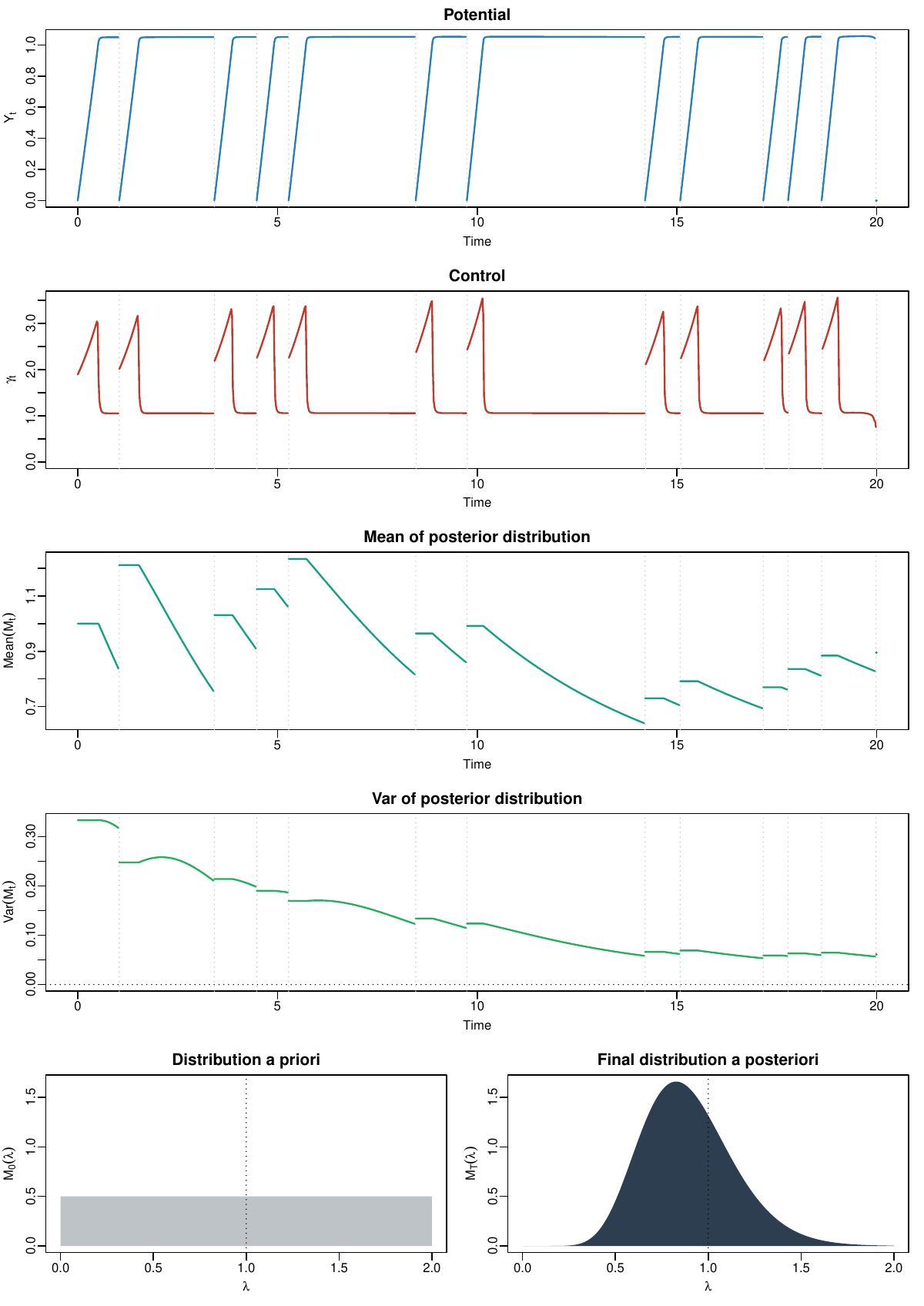}
\caption{An optimal trajectory with intensity approximating $\lambda \mathbf{1}_{\{y \geq 1\}}$ and prior $\mathcal{U}\left([0, 2]\right)$. The true value is $\lambda = 1$.}
\label{ind_unif}
\end{figure}

In Figure~\ref{ind_unif}, we present a simulated path of the optimal strategy with $\Lambda = 1$. The plots illustrate the same quantities as in Figure~\ref{x2_dirac}, except that the bottom-left and bottom-right panels display $m_0$ and $m_T$ as density functions, respectively.

Throughout the trajectory, the controller applies a control $t \mapsto \gamma_{t}$ that increases until the potential reaches slightly above $Y_t = 1$ at which point the control jumps to approximately $\gamma_t = 1$. Since the drift of the potential is $-Y_{t} + \gamma_{t}$, setting $\gamma_{t} = 1$ when the potential reaches 1 is optimal, as exceeding this value is unnecessary. Near the terminal time, after the final jump, the control decreases, as further increases in potential are not cost-effective.

The prior is initially centered at the true value. While the potential remains below 1, the posterior distribution is nearly unchanged, as no information about $\Lambda$ is gained. Once the potential reaches 1, the mean and variance decrease linearly. Upon observing a jump, the mean increases, while the variance's behavior depends on the observation specifics but generally decreases over time. Although the prior is symmetric around $\lambda = 1$, a jump excludes $\lambda = 0$, resulting in zero density at that point. The density at the upper boundary, $\lambda = 2$, remains positive but close to zero.

\FloatBarrier

\section*{Acknowledgments}

Nicolas Baradel acknowledges the financial support from the \emph{Fondation Natixis}. 
Quentin Cormier thanks Carl Graham for insightful discussions at various stages of this project, as well as Charles Bertucci and Mete Soner for their help with viscosity solutions.
\bibliographystyle{abbrvnat}
\bibliography{biblio}
\end{document}